\documentclass[11pt]{article}
\usepackage{amssymb}
\usepackage{amsmath}
\usepackage{amsthm}
\usepackage{hyperref}
\usepackage{amsfonts}
\usepackage{url}
\usepackage{color}
\usepackage{epsfig}
\usepackage{bm}
\usepackage{accents}
\usepackage{empheq}
\graphicspath{{Figs/}}

\newtheoremstyle{thmm}{1.5ex plus 1ex minus .2ex}{1.5ex plus 1ex minus.2ex}
{\rmfamily}{}{\bfseries}{}{1em}{}
\theoremstyle{thmm}
\newtheorem{theorem}{Theorem}[section]
\newtheorem{lemma}{Lemma}[section]

\newtheorem{remark}{Remark}[section]
\renewenvironment{proof}[1][Proof]{\noindent\textit{#1.} }{\hfill$\square$=}
\newtheorem{example}{Example}[section]

\allowdisplaybreaks \textwidth  6.0in \textheight 9in
\newcommand{\nn}{\nonumber}

\def\refe#1{(\ref{#1})}

\makeatletter            

\newcommand{\Rmnum}[1]{\expandafter\@slowromancap\romannumeral #1@}
\makeatother

\newcommand{\bA}{\mathbf{A}}
\newcommand{\bB}{\mathbf{B}}
\newcommand{\bC}{\mathbf{C}}
\newcommand{\bu}{\mathbf{u}}
\newcommand{\bv}{\mathbf{v}}

\newcommand{\D}{\mathrm{div}}
\newcommand{\C}{\mathbf{curl}}
\newcommand{\curl}{\mathrm{curl}}

\newcommand{\n}{\mathbf{n}}

\newcommand{\bH}{\mathbf{H}}
\newcommand{\bQ}{\mathbf{Q}}

\newcommand{\bX}{\mathbf{X}}

\begin{document}

\date{\today}

\title{\bf A linearized energy preserving finite element method
  for the dynamical incompressible magnetohydrodynamics equations}
\author{
  \setcounter{footnote}{0}
  Huadong~Gao
  \footnote{
    School of Mathematics and Statistics, 
    Huazhong University of Science and Technology, 
    Wuhan 430074, People¡¯s Republic of China.
    {\tt huadong@hust.edu.cn}.
    The work of the author was supported in part by a grant
    from the National Natural Science Foundation of China (NSFC) under
    grant No. 11501227.
  }
  \quad and \quad Weifeng~Qiu
  \footnote{ 
    Department of Mathematics,
    City University of Hong Kong, 83 Tat Chee Avenue, Hong Kong, China.
    {\tt weifeqiu@cityu.edu.hk}.
    The work of the author was supported by a grant
    from the Research Grants Council of the Hong Kong Special
    Administrative Region, China. (Project No. CityU 11304017)  
  }
}
\maketitle

\begin{abstract}
  We present and analyze a linearized
  finite element method (FEM) for the dynamical incompressible
  magnetohydrodynamics (MHD) equations.
  The finite element approximation is based on mixed conforming elements,
  where Taylor--Hood type elements are used for the Navier--Stokes equations
  and {N\'ed\'elec} edge elements are used for the magnetic equation.
  The divergence free conditions are weakly satisfied at the discrete level.
  Due to the use of {N\'ed\'elec} edge element,
  the proposed method is particularly suitable for problems defined on
  non-smooth and multi-connected domains.
  For the temporal discretization, we use a linearized scheme
  which only needs to solve a linear system at each time step.
  Moreover, the linearized mixed FEM is energy preserving.
  We establish an optimal error estimate under a very low
  assumption on the exact solutions and domain geometries.
  Numerical results which includes a benchmark lid-driven cavity
  problem are provided to show its effectiveness
  and verify the theoretical analysis.
  
\end{abstract}


\section{Introduction}
\setcounter{equation}{0}

In this paper, we consider the dynamical incompressible
magnetohydrodynamics (MHD) equations,
which is a coupled equation system of Navier--Stokes equations
of fluid dynamics and Maxwell's equations of electromagnetism
via Lorentz's force and Ohm's Law, as follows
\begin{empheq}[left=\empheqlbrace]{align} 
  & \frac{\partial \bu}{\partial t} - \frac{1}{R_e}\Delta \bu
  + \bu \cdot \nabla \, \bu
  + \nabla p - S_c \, \C \, \bB \, \times\bB =\mathbf{f},
  && \bm{x} \in \Omega,
  \label{pde-1}\\[4pt]
  & \frac{\partial \bB}{\partial t}
  + \frac{S_c}{R_m} \C \, (\C \, \bB) - S_c \C \, (\bu \times \bB)
  = \mathbf{0},
  && \bm{x} \in \Omega,
  \label{pde-2}\\[4pt]
  & \nabla \cdot \bu=0,
  && \bm{x} \in \Omega,
  \label{pde-3}\\[4pt]
  & \nabla \cdot \bB=0,
  && \bm{x} \in \Omega,
  \label{pde-4}
\end{empheq}
with boundary conditions
\begin{empheq}[left=\empheqlbrace]{align} 
  & \bu = \mathbf{0},
  \quad \bm{x} \in \partial \Omega,
  \label{bc-u} \\
  & \bB \cdot \n = 0, \quad
  \C \,\bB \times \n ={\mathbf{0}},
  \quad \bm{x} \in  \partial \Omega,
  \label{bc-mag}
\end{empheq}
and initial conditions
\begin{align}
  \bu(\bm{x},0) = \bu_0(\bm{x}),\quad
  \bB(\bm{x},0) = \bB_0(\bm{x}),
  \label{initial}
\end{align}
where $\D \, \bu_0(\bm{x}) = \D \, \bB_0(\bm{x}) = 0$.
In the above dynamical incompressible MHD system,
$\bu$ represents the velocity of the fluid flow,
$\bB$ represents magnetic field, $p$ represents the pressure
and $\mathbf{f}$ stands for the external body force term, respectively.
In this paper we assume $\Omega$ is a bounded Lipschitz polyhedral
domain in $\mathbb{R}^3$ (polygonal domain in $\mathbb{R}^2$),
which might not be convex or simply-connected.
The dynamical incompressible MHD equation system
is characterized by three parameters:
the hydrodynamic Reynolds number $R_e$,
the magnetic Reynolds number $R_m$ and the coupling number $S_c$.
Another commonly used boundary condition for equation \refe{pde-2} of $\bB$
is defined by
\begin{align}
& \bB \times \n = \mathbf{0}, \quad
  \quad \bm{x} \in  \partial \Omega.
  \label{bc-mag-alternate}
\end{align}
In addition, the boundary condition for $\bB$ might be of a mixed type,
i.e., \refe{bc-mag} is used on part of $\partial \Omega$,
while \refe{bc-mag-alternate} is used on the other part.
However, we shall mainly consider the boundary condition  \refe{bc-u}-\refe{bc-mag}
in this paper. 

The dynamical incompressible MHD equations have attracted an amount of attention
due to its important applications in modeling liquid
metals \cite{Gerbeau-LL,Moreau} and plasma physics \cite{Goedbloed-Poedts}.
We refer to the monograph \cite{Gerbeau-LL} by Gerbeau, Le Bris and Leli\`evre
as a summary of recent progress for the MHD equations,
which includes mathematical modeling, analysis and numerical methods.
Mathematical analyses of the dynamical MHD model can be found in
\cite{CorMar,Lin-Zhang,Schonbek-Soli,Sermane-Temam} and reference therein.
Global existence of weak solutions has been well established.
Sermane and Temam proved the existence and uniqueness of local strong solution
on regular domains \cite{Sermane-Temam}.
There have been numerous works on
numerical methods for the incompressible MHD equations,
see \cite{Badia-Codina-Planas, GreifSW, GunzburgerMP, Hiptmair_LMZ, HuMaXu,
  Marioni,  Phillips, Schotzau, Zhang-Yang-Bi}.
Due to the nonlinear coupling of the unknowns,
the divergence free constraints and the low regularity of the exact solutions,
it is a challenging task to design efficient numerical schemes for the
dynamical incompressible MHD equations.
Several nonlinear schemes have been suggested, see \cite{Hiptmair_LMZ,HuMaXu},
where the motivation for using nonlinear scheme is to
preserve the properties of the original equations as much as possible.
For example, energy conservation can be preserved at the discrete level
for nonlinear schemes studied in \cite{Hiptmair_LMZ,HuMaXu}.
However, a key observation on the two nonlinear terms 
\begin{align}
  - S_c \, \C \, \bB \, \times\bB \quad
  {\small \textrm{in \refe{pde-1}}},  \quad
  - S_c \C \, (\bu \times \bB) \quad
  {\small \textrm{in \refe{pde-2}}}, \quad
  \nn
\end{align}
tells that the weak formulation admit anti-symmetric structure.
That is, the sum of the corresponding terms in the weak formulation
\begin{align}
  - S_c(\C \, \bB \, \times\bB \, , \, \bv) , \quad
  - S_c (\bu \times \bB, \C \, \bC) , \quad
  \label{anti-symmetric-property}
\end{align}
vanishes if we take $\bv = \bu$ and $\bC = \bB$.
By using this property,
a careful linearization yields an energy preserving discretization,
see the FEM system \refe{fem-1}-\refe{fem-3}.
We shall mention that \cite{Prohl,Zhang-Yang-Bi}
also noticed this anti-symmetric structure in incompressible MHD equations.
Another issue is the constraint $\D \, \bB =0$.
For the ideal MHD flow problem,
not guaranteeing $\D \, \bB =0$ to round-off error
may lead to nonphysical solutions,
see the numerical report in \cite{Brackbill} by Brackbill and Barnes.
Motivated by this,
some attempts in developing divergence free numerical methods
can be found in \cite{Hiptmair_LMZ, HuMaXu}.
Hu, Ma and Xu suggested a mixed finite element method (FEM) in \cite{HuMaXu},
where the current $\mathbf{E} = \C \, \bB - \bu \times \bB$ was introduced
as a new variable.
The $\bH(\C)$ conforming {N\'ed\'elec} edge element is applied for the
discretization of $\mathbf{E}$.
By using $\bH(\D)$ conforming
Raviart--Thomas element to approximate $\bB$,
the authors in \cite{HuMaXu} proved that their scheme
is divergence free on each element.
It seems that the only disadvantage for this approach is
the expensive computational costs, where vector elements are heavily used
in the spatial discretization.
For problems defined on simply-connected domain,
another strategy to eliminate the divergence free constraint
is to introduce a new variable $\bA$ such that $\bB = \C \, \bA$
and then solve the new dynamical MHD system of $\bA$, $\bu$ and $p$,
see the scheme proposed in \cite{Hiptmair_LMZ} by Hiptmair et al..
However, the ideal MHD model in \cite{Brackbill} is essentially different with
the incompressible MHD equations \refe{pde-1}-\refe{initial}.
To the best knowledge of the authors, there are no numerical reports
which show that $\D \bB_h \neq 0$ leads to nonphysical solutions
for \refe{pde-1}-\refe{initial}.
An alternative and less expensive way to deal with this constraint
is to enforce the divergence free condition weakly.
For example, in \cite{Schotzau} the numerical solution $\bB_h$ satisfies
\[
\int_{\Omega} \, \bB_h \,\cdot\, \nabla v_h \, \mathrm{d}x=0, \quad
\forall v_h \in V_h(\Omega)
\]
where $V_h(\Omega)$ denotes a certain FE space.
Let us remark that this approach and its generalizations has been widely used,
see \cite{Badia-Codina-Planas, Codina, GreifSW, GunzburgerMP, HaslerSS, He,
Phillips, Prohl, Schotzau, Zhang-Yang-Bi}.
Finally, it should be noted that the regularity of $\bB$
on non-convex and non-smooth domains is lower than $\bH^1(\Omega)$.
Due to the $\C \C$ structure of magnetic equation \refe{pde-2},
it is well-known that on a general non-convex Lipschitz polyhedron
\cite{Amrouche-BDG,Costabel_Dauge},
$\bH(\D,\Omega) \cap \bH(\C, \Omega)$
is embedded in $\bH^s(\Omega)$ for $\frac{1}{2} < s < 1$ only.
However, conventional Lagrange FEMs for a scalar parabolic equation
require the regularity of the exact solution in $H^{1+s}(\Omega)$ with 
$s>0$ \cite{CLTW,CHou}.
Thus, analyses in \cite{He, Zhang-Yang-Bi} all assume that
the domain $\Omega$ is convex or smooth
to ensure the convergence of conventional Lagrange FEMs.
To overcome this difficulty,
the $\bH(\C)$ conforming {N\'ed\'elec} edge element
can solve for the $\C\,\C$ problem correctly on more general geometries,
which has attracted much attention in the mathematical society
and been successfully used in electromagnetics industry.

In this paper, we present a linearized mixed FEM
for the dynamical incompressible MHD equations.
Though the incompressible MHD equations \refe{pde-1}-\refe{initial}
introduce strongly nonlinear coupling between $\bu$ and $\bB$,
a careful linearization for $ - S_c \, \C \, \bB \, \times\bB$ in \refe{pde-1}
and $- S_c \C \, (\bu \times \bB)$ in \refe{pde-2}
yields an energy preserving scheme,
see the linearized FEM \refe{fem-1}-\refe{fem-3}
in section \ref{chap-scheme}.
The cancellation of these two nonlinear terms plays a key role in later
theoretical analysis. Moreover, the proposed scheme is linear.
At each time step, we only need to solve a linear system
which makes the proposed scheme very attractive in practical computations.
Limited work has been done for problems defined on non-convex domains
when $\bB$ is not in $\bH^1(\Omega)$.
By using {N\'ed\'elec} edge element space, we obtain optimal
error estimate for a linearized scheme
under relatively low regularity assumption on the exact solution.
In particular, the proposed scheme converges and the error analysis holds
for problems defined on non-smooth, non-convex and multi-connected domains.

The rest of this paper is organized as follows. 
In section \ref{chap-scheme}, we provide a linearized
mixed FEM for the incompressible MHD equations.
In section \ref{chap-analysis}, we present analysis of
the linearized scheme.
Numerical examples for both two- and three-dimensional models 
are given in section \ref{numer-section} to show the efficiency of our method.
Some concluding remarks are give in section \ref{numer-conclusion}.

\section{A conservative linearized mixed FEM}
\label{chap-scheme}
\setcounter{equation}{0}

\subsection{Preliminaries}
\label{preliminaries}

We now introduce notions for some standard Sobolev spaces.
For any two functions $u$, $v \in {L}^{2}(\Omega)$, 
we denote the ${L}^{2}(\Omega)$ inner product and norm by
\begin{equation}
  (u, \,v)  = \int_{\Omega} u(x) \cdot v(x) \, {\mathrm{d}} x, 
  \qquad {\left \| u \right \|_{L^2}} = (u,u)^{\frac{1}{2}} \, ,
  \nn
\end{equation}
where $\cdot$ denotes the inner product in case of vectorial functions.
Let $W^{k,p}(\Omega)$ be the Sobolev space defined on $\Omega$, 
and by conventional notations, $H^{k}(\Omega) := W^{k,2}(\Omega)$,
$\accentset{\circ}{H}^{k}(\Omega) := \accentset{\circ}{W}^{k,2}(\Omega)$.
Let $\bH^{k}(\Omega) = [H^{k}(\Omega)]^{d}$ 
be a vector-valued Sobolev space,
where $d$ is the dimension of $\Omega$.
For a positive real number $s = k + \theta$ with $0 < \theta < 1$,
we define $H^s(\Omega)= (H^k, H^{k+1})_{[\theta]}$ by the complex interpolation,
see \cite{Bergh_Lofstrom}.
For $(\bu, p)$, we shall introduce
\begin{align}
  & \accentset{\circ}{\bH}^{1}(\Omega)
  = \left \{ \bu \in \bH^{1}(\Omega), \quad \bu|_{\partial \Omega}
  = \mathbf{0}  \right \} \,,
  \quad
  & L_0^2 = \left\{ u \in L^2, \quad  (u, 1)=0 \right \}.
\end{align}
For the magnetic field $\bB$, we denote 
\begin{align}
  {\bH}(\C) = \left\{
  \bB \, \big| \, \bB \in \mathbf{L}^{2}(\Omega),
  \C \, \bB  \in \mathbf{L}^2(\Omega) \right\} \,
  \textrm{with $\| \bB \|_{{\bH}(\C)} = \left(\| \bB \|_{{L}^2}^2 
    + \|\C \, \bB\|_{{L}^2}^2  \right)^{\frac{1}{2}} $} .
  \nn
\end{align}
and its dual space ${{\bH}(\C)}'$ with norm 
\begin{align}
  \left \| \mathbf{B} \right \|_{{{\bH}(\C)}'}
  := \sup_{\mathbf{w} \in {{\bH}(\C)}}
  \frac{(\mathbf{B} \, , \, \mathbf{w})}
       {\left \| \mathbf{w} \right \|_{{\bH}(\C)}}\,.
       \nn
\end{align}
Moreover, we denote 
\begin{align}
  {\bH}(\D) = \left\{
  \bA \, \big| \, \bA \in \mathbf{L}^{2}(\Omega),
  \D \, \bA  \in L^2(\Omega) \right\} \,
  \textrm{with $\|\bA\|_{{\bH}(\D)} = \left(\|\bA\|_{{L}^2}^2 
    + \|\D \, \bA\|_{{L}^2}^2  \right)^{\frac{1}{2}} $}
  \nn
\end{align}

Now we introduce some notations for the numerical methods.
Let $\mathcal{T}_{h}$ be a quasi-uniform
tetrahedral partition of $\Omega$ with $\Omega = \cup_{K} \Omega_{K}$.
The mesh size is denoted by  $h = \max_{\Omega_{K} \in \mathcal{T}_{h}} 
\{ \mathrm{diam} \, \Omega_{K}\}$.
To approximate $(\bu, p)$, we use the finite element pair 
\[
\bX_{h} \subset \accentset{\circ}{\bH}^{1}(\Omega)
\qquad \textrm{and} \qquad
M_h  \subset {L}^{2}_0(\Omega),
\]
which satisfies the discrete inf-sup condition:
there exits a constant $\beta > 0$ such that
\begin{align}
  \inf_{0 \neq q_h \in M_h} \,
  \sup_{\mathbf{0} \neq \bv_h \in \bX_{h}} \,
  \frac{(q_h \,,\, \D \, \bv_h)}{\|\bv_h\|_{\bH^1} \, \|q_h\|_{L^2}}
  \geq \beta_1,
  \label{LBB-ns}
\end{align}
where $\beta_1$ depends on $\Omega$ only.
In this paper, we choose the popular generalized Taylor--Hood
FE space $\bX_{h}^{k+1} \times M_h^{k}$ with $k \ge 1$ for the approximation
of $(\bu, p)$, see \cite{Boffi-Brezzi-Fortin, Girault-Raviart}.
Here $\bX_{h}^{k+1}$ is the $(k+1)$-th order vectorial Lagrange FE
space and $M_h^{k}$ represents the $k$-th order scalar Lagrange FE subspace
of $L_0^2(\Omega)$, respectively.
We shall also introduce $V_h^{k}$ to be the
$k$-th order scalar Lagrange FE subspace of $H^1(\Omega)$.
To approximate $\bB$, we denote by ${\bQ}_{h}^{k}$ the $k$-th 
order first type {N\'ed\'elec}  FE subspace of 
$\bH(\C)$, where the case $k=1$ corresponds to 
the lowest order {N\'ed\'elec}  edge element ($6$ dofs). 
We denote by $\Pi_h$ a general projection operator on $\bX_{h}^{k}$,
$M_{h}^{k}$, $\bQ_{h}^{k}$ and $V_{h}^{k}$. The approximation
properties of $\Pi_h$ are summarized in the following lemma.

\begin{lemma}
  \label{lemma-interpolation}
  By noting the approximation properties of the finite element spaces
  $\bX_{h}^{k}$, $M_{h}^{k}$(or $V_{h}^{k}$)  and  $\bQ_{h}^{k}$,
  we denote by $\Pi_h$ the projection operator on  $\bX_{h}^{k}$,
  $M_{h}^{k}$, and $\bQ_{h}^{k}$, satisfying 
  \begin{equation}
  \left \{
  \begin{array}{ll}
  {\left\|\bm{\omega} - \Pi_h \bm{\omega} \right\|}_{L^{2}}
  \le C h^s {\left\| \bm{\omega} \right\|}_{\bH^{s}} \, , 
  & 0 < s \le k+1 \, ,
  \\[4pt]
  {\left\|\omega - \Pi_h \omega \right\|}_{L^{2}}
  \le C h^s {\left\| \omega \right\|}_{H^{s}} \, , 
  & 0 < s \le k+1 \, ,
  \\[4pt]
  {\left\|\bm{\chi}  - \Pi_h \bm{\chi}\right\|}_{L^{2}} +
  {\left\|\C \, (\bm{\chi} - \Pi_h \bm{\chi})\right\|}_{L^{2}} 
  \le C h^s ({\left\|\bm{\chi}\right\|}_{H^{s}} +
  {\left\|\C \, \bm{\chi}\right\|}_{H^{s}}), 
  & \frac{1}{2} < s \le k \, .
  \end{array}
  \right.
  \label{interpolate}
  \end{equation}
\end{lemma}
The interpolation results for Lagrange element space can be found in \cite{BS}.
We refer to \cite[Theorem 5.41]{Monk} and \cite{Alonso_Valli} for the proof of
the interpolation onto the {N\'ed\'elec} edge element space $\bQ_{h}^{k}$.
For the time discretization,
let ${\left \{ t_{n} \right \}}_{n=0}^{N}$ be a uniform partition
in the time direction with the step size $\tau = \frac{T}{N}$,
and let $u^n = u(\cdot, n \tau)$. 
For a sequence of functions $\{U^{n}\}_{n=0}^{N}$ defined on $\Omega$, 
we denote 
\begin{eqnarray}
  {D_{\tau}} U^{n} = \frac{U^{n}-U^{n-1}}{\tau}, \quad
  \overline{U}^{n} = \frac{U^{n}+U^{n-1}}{2}, \quad
\textrm{for $n=1$, $2$, $\ldots$, $N$}.
\nn
\end{eqnarray}

\subsection{A linearized mixed FEM}
\label{a-mixed-fem}

With the above notations, a linearized backward Euler mixed FEM
for dynamical incompressible MHD equations \refe{pde-1}-\refe{bc-mag}
is to look for
$(\bu_h^n, \bB_h^n, p_h^n) \in \bX_h^{k+1} \times \bQ_h^{\widehat{k}} \times M_h^k$,
such that for any
$(\bv_h^n, \bC_h^n, q_h^n) \in \bX_h^{k+1} \times \bQ_h^{\widehat{k}} \times M_h^k$
\begin{empheq}[left=\empheqlbrace]{align} 
  & (D_\tau \bu_h^n, \bv_h)
  + \frac{1}{R_e} (\,\nabla\, \overline{\bu}_h^n \,,\, \nabla \bv_h)
  + \frac{1}{2} \big[ ( \bu_h^{n-1} \cdot \nabla \, \overline{\bu}_h^n, \bv_h)
    - (\bu_h^{n-1} \cdot \nabla \, \bv_h \,,\, \overline{\bu}_h^n ) \big]
  \nn \\
  & \qquad \qquad \qquad \qquad \qquad
  - ({p}_h^n, \nabla \cdot \bv_h)
  - S_c (\C \, \overline{\bB}_h^{n} \, \times \bB_h^{n-1} \, , \, \bv_h)
  = (\mathbf{f}^{n}, \bv_h),
  \label{fem-1}\\[3pt]
  & ( D_\tau \bB_h^n , \bC_h )
  + \frac{S_c}{R_m} (\C \, \overline{\bB}_h^n, \C \, \bC_h)
  - S_c (\overline{\bu}_h^n \times \bB_h^{n-1}, \C \, \bC_h) = 0\,,
  \label{fem-2} \\[3pt]
  & (\nabla \cdot \overline{\bu}_h^n \, , \, q_h) = 0,
  \label{fem-3}
\end{empheq}
where $k$, $\widehat{k} \ge 1$.
At the initial time step,  $\bu_h^n = \Pi_h \bu_0$ and $\bB_h^n = \Pi_h \bB_0$.
We give several remarks concerning the proposed scheme.

\begin{remark}
  \label{remark-matrix-form}
  The above linearized mixed FEM \refe{fem-1}-\refe{fem-3}
  can be written in matrix form as
  \begin{align}
    \left[
      \begin{array}{ccc}
        \frac{1}{\tau} \, {\mathsf{M}_1} + \frac{1}{2 R_e} \mathsf{K}_1
        + \frac{1}{2}\mathsf{N}_1(\bu_h^{n-1})
        & \mathsf{B} & \frac{S_c}{2}\, \mathsf{N}_2(\bB_h^{n-1}) \\[8pt]
        \mathsf{B}^T & 0 & 0  \\[8pt]
        \frac{S_c}{2} \, \mathsf{N}_2^T(\bB_h^{n-1}) & 0 & 
        \frac{1}{\tau} \, {\mathsf{M}_2} + \frac{S_c}{2 R_m} \mathsf{K}_2 
      \end{array}
      \right]
    \left[
      \begin{array}{c}
        \bu_h^{n} \\[9pt]
        p_h^{n} \\[9pt]
        \bB_h^{n}
      \end{array}
      \right]
    := {\mathsf{A}} {\mathsf{x}}
    \, = \,{\mathsf{b}}\,,
    \label{matrix-form}
  \end{align}
  with $(p_h^{n},1)=0$.
  In terms of basis functions $\{\phi\}_{i=1}^{N_h}$,
  the block matrices in \refe{matrix-form} are generated by
  \begin{align}
  \begin{array}{rl}
  \{\mathsf{M}_1\}_{i,j}: & 
  ({\phi}_j, {\phi}_i),   
  \quad {\phi}_i, {\phi}_j \in \bX_{h}^{k} , \\[4pt]
  \{\mathsf{K}_1\}_{i,j}: & 
  (\nabla {\phi}_j, \nabla{\phi}_i),   
  \quad {\phi}_i, {\phi}_j \in \bX_{h}^{k} , \\[4pt]
  \{\mathsf{N}_1\}_{i,j}:  & 
  \frac{1}{2} \big[ ( \bu_h^{n-1} \cdot \nabla \, \phi_j, \phi_i)
    - (\bu_h^{n-1} \cdot \nabla \, \phi_i, \phi_j ) \big],
  \quad {\phi}_i, {\phi}_j \in \bX_{h}^{k} , \\[4pt]
  \mathsf{B}_{i,j}:  & 
  -( \phi_j \, \nabla \cdot \phi_i),
  \qquad \phi_j \in {W}_{h}^{k}, \quad {\phi}_i \in \bX_{h}^{k} ,\\[4pt]
  \{\mathsf{N}_2\}_{i,j}:  &
  -(\C \, \phi_j \, \times \, \bB_h^{n-1}, \phi_i),
  \qquad {\phi}_j \in {\bQ}_h^k,
  \quad {\phi}_i \in \bX_{h}^{k} ,\\[4pt]
  \{\mathsf{M}_2\}_{i,j}: & 
  ({\phi}_j, {\phi}_i),   
  \quad {\phi}_i, {\phi}_j \in {\bQ}_{h}^{\widehat{k}} ,\\[4pt]
  \{\mathsf{K}_2\}_{i,j}: & 
  (\C \, {\phi}_j, \C \, {\phi}_i),   
  \quad {\phi}_i, {\phi}_j \in {\bQ}_{h}^{\widehat{k}} .
  \end{array} 
  \nn
  \end{align}
  To prove the existence and uniqueness of the linearized mixed FEM,
  it suffices to show that
  $\mathsf{A} \mathsf{x} = \mathsf{0}$ admits zero solution only.
  Assuming $\mathsf{A} \mathsf{x} = \mathsf{0}$, then we have
  \begin{align}
    0 = \mathsf{x}^T \mathsf{A} \mathsf{x} =
    \frac{1}{\tau}\|\bu_h^n\|_{L^2}^2
    + \frac{1}{2} R_e^{-1} \|{\bu}_h^n\|_{L^2}^2
    + \frac{1}{\tau}\|\bB_h^n\|_{L^2}^2
    + \frac{1}{2} S_c R_m^{-1} \|\C \, {\bB}_h^n\|_{L^2}^2 \,,
    \label{unique-solution}
  \end{align}
  from which $\bu_h^n = \bB_h^n = \mathbf{0}$ follows directly.
  By using the inf-sup conditions \refe{LBB-ns},
  one can deduce that $p_h^n=0$, which immediately leads to
  the fact that $\mathsf{A}$ is invertible. 
\end{remark}

\begin{remark}
  In the above scheme, we only consider the homogeneous boundary condition
  \refe{bc-u}-\refe{bc-mag}. However, it should be noted that the proposed
  scheme is able to deal with the mixed type boundary condition for $\bB$
  conveniently.
  For instance, we assume $\partial \Omega = \Gamma_1 \cup \Gamma_2$ where
  $\Gamma_1 \cap \Gamma_2 = \emptyset$.
  The boundary condition for $\bB$ is set to be
  $\bB \times \n = \mathbf{0}$ on $\Gamma_1$, while $\bB \cdot \n = {0}$
  and $\C \, \bB \times \n = \mathbf{0}$ on $\Gamma_2$.
  In this case, the FE space for $\bB$ shall be
  $\{\bB_h \in {\bQ}_{h}^{\widehat{k}} \, | \,
  \bB_h \times \mathbf{n} = \mathbf{0} \, \textrm{on ${\Gamma_1}$} \}$.
  The stability and error analyses also
  hold for problems with mixed type boundary conditions.
\end{remark}

\begin{remark}
  The proposed scheme \refe{fem-1}-\refe{fem-3}
  satisfies a weakly divergence free property, namely,
  for any $s_h \in V_h^k$, taking $\bC_h^n = \nabla s_h$ in \refe{fem-2} gives
  \begin{align}
    ( D_\tau \bB_h^n ,  \nabla s_h ) = 0, \quad
    \textrm{for $n=1$, $\ldots$, $N$,}
    \nn
  \end{align}    
  which in turn leads to
  \begin{align}
    ( \bB_h^n ,  \nabla s_h ) = 0, \quad
    \forall s_h \in V_h^k
    \label{weakly-div0}
  \end{align}
  provided that $( \bB_h^0 ,  \nabla s_h ) = 0$.
  Here, we have used the fact that $\nabla V_h^k \subset {\bQ}_{h}^{k}$.
\end{remark}

We present the Gagliardo--Nirenberg inequality and the discrete Gronwall's
inequality in the following lemmas which will be
frequently used in our proofs.
\begin{lemma}
\label{GN}
{ (Gagliardo--Nirenberg inequality \cite{Nirenberg}):
Let $u$ be a function defined on $\Omega$ in $\mathbb{R}^d$
and $\partial ^{s} u$ be any partial derivative of $u$ of order $s$, then
\begin{equation}
\|\partial ^{j} u\|_{L^p}
\le C \|\partial^{m} u\|_{L^k}^{a} \, \|u\|_{L^q}^{1-a}
+ C \|u\|_{L^q},
\nn
\end{equation}
for $0 \le j < m$ and $\frac{j}{m} \le a \le 1$ with
\[
\frac{1}{p} =
\frac{j}{d} + a \left( \frac{1}{r} - \frac{m}{d}\right)
+(1-a) \frac{1}{q} \, ,
\]
except $1 < r < \infty$ and $m-j-\frac{d}{r}$ is a non-negative
integer, in which case the above estimate holds only for
$\frac{j}{m} \le a < 1$. }
\end{lemma}
\begin{lemma}
\label{gronwall} Discrete Gronwall's inequality
{\cite{Heywood_Rannacher}} : Let $\tau$, $B$ and $a_{k}$, $b_{k}$,
$c_{k}$, $\gamma_{k}$, for integers $k \geq 0$, be non-negative
numbers such that
\[
a_{J} + \tau \sum_{k=0}^{J} b_{k} \leq \tau \sum_{k=0}^{J}
\gamma_{k} a_{k} + \tau \sum_{k=0}^{J} c_{k} + B \, , \quad
\mathrm{for } \quad J \geq 0 \, ,
\]
suppose that $\tau \gamma_{k} < 1$, for all $k$,
and set $\sigma_{k}=(1-\tau \gamma_{k})^{-1}$. Then
\[
a_{J} + \tau \sum_{k=0}^{J} b_{k} \leq  \exp(\tau \sum_{k=0}^{J}
\gamma_{k} \sigma_{k}) (\tau \sum_{k=0}^{J} c_{k} + B) \, , \quad
\mathrm{for } \quad J \geq 0 \, .
\]
\end{lemma}

\section{Analysis of the linearized mixed FEM}
\label{chap-analysis}
\setcounter{equation}{0}

\subsection{Stability Analysis}

\begin{theorem}
\label{thm-stability-pde}
For any $(\bu, p,\bB)$ that satisfy the dynamical incompressible
MHD equations \refe{pde-1}-\refe{bc-mag}, the following estimate holds
\begin{align}
  \frac{1}{2} \frac{d}{d t} \|\bu\|_{L^2}^2
  + \frac{1}{2} \frac{d}{d t} \|\bB\|_{L^2}^2  
  + \frac{1}{R_e} \|\nabla \bu\|_{L^2}^2
  + \frac{S_c}{R_m} \|\C \, \bB\|_{L^2}^2
  = (\mathbf{f},\bu).
  \label{energy-estimate-pde}
\end{align}
\end{theorem}

\begin{proof} A standard energy estimate yields the desired results. \end{proof}

For the linearized FEM equations \refe{fem-1}-\refe{fem-3},
we can prove the following theorem,
which can be viewed as the discrete version
of Theorem \ref{thm-stability-pde}.
\begin{theorem}
\label{thm-stability-fem}
The numerical solutions $(\bu_h^n, p_h^n, \bB_h^n)$ to
the linearized mixed FEM \refe{fem-1}--\refe{fem-3}
satisfy the following energy preserving property
\begin{align}
  & \|\bu_h^n\|_{L^2}^2 + \|\bB_h^n\|_{L^2}^2
  + 2 \tau \frac{1}{R_e} \|\nabla \overline{\bu}_h^n \|_{L^2}^2
  + 2 \tau \frac{S_c}{R_m} \|\C \, \overline{\bB}_h^n \|_{L^2}^2
  \nn \\
  & = \|\bu_h^{n-1}\|_{L^2}^2 + \|\bB_h^{n-1}\|_{L^2}^2 +
  2 \tau (\mathbf{f}^{n} \,,\, \overline{\bu}_h^n) \,,
  \label{energy-estimate-fem}
\end{align}
which further results in the following energy stability
\begin{align}
  & \max_{0 \le j \le n}\left(\|\bu_h^j\|_{L^2}^2+\|\bB_h^j\|_{L^2}^2\right)
  + \sum_{j=1}^n \tau \left(\frac{1}{R_e} \|\nabla \, \overline{\bu}_h^j \|_{L^2}^2
  + \frac{S_c}{R_m} \|\C \, \overline{\bB}_h^j \|_{L^2}^2 \right)
  \nn \\
  & \le \|\bu_h^{0}\|_{L^2}^2 + \|\bB_h^{0}\|_{L^2}^2 +
  C \sum_{j=1}^n \tau \|\mathbf{f}^{j+1}\|_{H^{-1}}^2.
  \label{energy-estimate-fem-further}
\end{align}
\end{theorem}
\begin{proof}
  By taking $\bv_h = \overline{\bu}_h^n$ into \refe{fem-1},
  $\bC_h= \overline{\bB}_h^n$   into \refe{fem-2}
  and $q_h=p_h^n$ into \refe{fem-3}, respectively, and summing up
  the results, we obtain  \refe{energy-estimate-fem}.
  By using the discrete Gronwall's inequality, we can prove
  \refe{energy-estimate-fem-further}.
\end{proof}
\begin{remark}
  We shall note that Theorem \ref{thm-stability-fem} does
  not depend on the choice of the finite element space.
  If other boundary conditions are used, energy preserving
  property can be proved similarly.
\end{remark}

\subsection{Error analysis of the linearized FEM}
\setcounter{equation}{0}

To do the error estimate, we assume that the initial-boundary
value problem \refe{pde-1}-\refe{initial}
has a unique solution satisfying the regularity assumption below
\begin{equation} 
  \left \{
  \begin{array}{l}
      \bu \in {L^{\infty}(0,T;\bH^{{l+1}})} \, ,
      \bu_{t} \in {L^{\infty}(0,T;\bH^{{l+1}})} \, , 
      \bu_{tt} \in {L^2(0,T;\mathbf{L}^{2})}
      \\
      p \in {L^{\infty}(0,T;{H}^{{l}})} \,,
      p_{t} \in {L^{\infty}(0,T;H^{{l}})} \, , 
      \label{regularity-1}
  \end{array}
  \right.
\end{equation} 
and
\begin{equation} 
  \left \{
  \begin{array}{l}
    \bB \in {L^{\infty}(0,T;\mathbf{H}^{l})} \, , 
    \bB_{t} \in {L^{\infty}(0,T;\mathbf{H}^{l})} \, , 
    \bB_{tt} \in {L^2(0,T;\mathbf{L}^{2})} \, ,
    \\
    \C \, {\bB} \in {L^{\infty}(0,T;\mathbf{H}^{l})} \, , 
    \C \, {\bB}_{t} \in {L^{\infty}(0,T;\mathbf{H}^{l})} \, , 
    \label{regularity-2}
  \end{array}
  \right.
\end{equation} 
where $l > \frac{1}{2}$ depends on the regularity of the domain $\Omega$.
In the rest part of this paper, for simplicity of notation we denote
by $C$ a generic positive constant and $\epsilon$ a generic small
positive constant, which are independent of $j$, $h$ and $\tau$.
We present our main results on error estimates in the following theorem.

\begin{theorem}
\label{thm-main-error}
Suppose that the incompressible MHD system
\refe{pde-1}-\refe{initial} has a unique solution $(\bu, p, \bB)$
satisfying the regularity \refe{regularity-2}.
Then the linearized backward Euler
mixed FEM \refe{fem-1}-\refe{fem-3} admits a unique
solution $(\bu_h^n, p_h^n, \bB_h^n)  \in \bX_h^{k+1} \times \bQ_h^{\widehat{k}} \times M_h^k$
for $n=1$, $\ldots$, $N$,
and there exist two positive constants $\tau_0$ and $h_0$ such
that when $\tau < \tau_0$ and $h \le h_0$
\begin{align}
  & \max_{0 \leq n \leq N} \Big(
  {\| \bu_{h}^{n} - \bu^{n}\|_{L^2}^2}
  +{\| \bB_{h}^{n} - \bB^{n}\|_{L^2}^2} \Big)
  \nn \\
  & + \tau \sum_{m=1}^{N}
  \left( \| \nabla (\overline{\bu}_{h}^{m} - \overline{\bu}^{m}) \|_{L^2}^2
  + \| \C (\overline{\bB}_{h}^{m} - \overline{\bB}^{m}) \|_{L^2}^2
  \right)
  \leq C_* (\tau^{2} + h^{2s}) \, ,
  \label{main-error-results}
\end{align}
with $s = \min \{k, \widehat{k}, l\}$,
where $\widehat{k}$ and $k$ are
the order index of the finite element spaces,
$l$ is the index of regularity of the exact solutions.
In \refe{main-error-results},
$C_*$ is a positive constant independent of $n$, $h$ and $\tau$. 
\end{theorem}

\subsubsection{The Stokes projection and some error bounds}

To do error analysis,
we shall introduce the Stokes projection
$\mathbf{R}_h: (\accentset{\circ}{\bH}^1, L^2_0)
\rightarrow (\bX_h^{k+1}, M_h^{k})$.
For given $t \in (0,T]$, we look for
  $\mathbf{R}_h(\bu,p):= (\mathbf{R}_h(\bu,p)_1,\mathbf{R}_h(\bu,p)_2)
  \in (\bX_h^{k+1}, M_h^{k})$
such that
\begin{align}
  \left\{
  \begin{array}{ll}
    \frac{1}{R_e}(\nabla (\mathbf{R}_h(\bu,p)_1 - \bu),\nabla \bv_h)
    -(\mathbf{R}_h(\bu,p)_2-p,\nabla \cdot \bv_h) = 0\, ,
    & \forall \bv_h \in \bX_h^{k+1},
    \\[4pt]
    (\nabla \cdot (\mathbf{R}_h(\bu,p)_1- \bu), q_h)=0 \, ,
    & \forall q_h \in M_h^k \,.
  \end{array}
  \right.
  \label{projector-stokes}
\end{align}
For simplicity, we denote
\[
R_h \bu = \mathbf{R}_h(\bu,p)_1,\quad
R_h p = \mathbf{R}_h(\bu,p)_2,\quad
\]
Theoretical analysis on convergence 
and stability of the above projections 
can be found in \cite{Girault-Raviart}.
We summarize the main results in the following lemma.

\begin{lemma}
  \label{projection-stokes-mag}
For the projections defined above,  
the following error estimates hold 
  \begin{align}
  & \|R_h \bu - \bu\|_{\bH^1} + \|R_h p - p \|_{L^2} 
    \le C \Big(
    \inf_{\bv_h \in \bX_h^{k+1}}\|\bv_h - \bu\|_{\bH^1}
    + \inf_{q_h \in M_h^k}\| q_h - p \|_{L^2} 
    \Big) \,,
    \label{elliptic-error1}
  \end{align}
We denote the projection errors of $(\bu, p)$ and
interpolation error of $\bB$ by
\begin{align}
  & \theta_{\bu} = R_h \bu - \bu, \quad
  \theta_{p} = R_h p - p, \quad
  \theta_{\bB} = \Pi_h \bB - \bB.
  \nn
\end{align}
Then, by the regularity assumption \refe{regularity-2} and
Lemma \ref{projection-stokes-mag}, we have 
\begin{align}
  & \left\|\theta_{\bu}\right\|_{H^1} + \left\|\theta_{p}\right\|_{L^2}
  \le C h^s ( \| \bu \|_{H^{1+s}} + \| p \|_{H^{s}} ),
  \label{projection-stokes-error} \\
  & \left\|\theta_{\bB}\right\|_{\bH(\C)} \le C h^s \| \bB \|_{\bH^s(\C)} 
  \label{projection-mag-error} 
\end{align}
and
\begin{align}
  & \left\|\frac{\partial \theta_{\bu}}{\partial t}\right\|_{H^s}
  \le C h^{s} \left(\left\|\frac{\partial \bu}{\partial t}\right\|_{H^{1+s}}
  + \left\| \frac{\partial p}{\partial t} \right\|_{H^{s}} \right),
  \nn \\
  & \left\|\frac{\partial \theta_{\bB}}{\partial t}\right\|_{L^2}
  \le C h^s \left( \left\|\frac{\partial \bB}{\partial t}\right\|_{H^s}
  + \left\|\C \, \frac{\partial \bB}{\partial t}\right\|_{H^s}
  \right).
\end{align}
Moreover, by using inverse inequalities
we can deduce the following uniform boundedness
for $R_h\bu$ and $\Pi_h\bB$
\begin{align}
  &\|R_h\bu \|_{^\infty} + \|R_h\bu \|_{W^{1,3}}
  \le C(\|\bu\|_{H^{1+l}} + \|p\|_{H^{l}}) \le C\, .
  \label{uniform-bound-u}
  \\
  &\|\Pi_h\bB \|_{L^3} + \|\C \, \Pi_h\bB \|_{L^3}
  \le C(\|\bB\|_{H^l}+ \|\C \,\bB\|_{H^l}) \le C.
  \label{uniform-bound-B}
\end{align}
where $l>\frac{1}{2}$ in the regularity assumption
\refe{regularity-1}-\refe{regularity-2}.
\end{lemma}
With the above projection error estimates we only
need to estimate the following error equations
\begin{align}
  & e_{\bu}^n = \bu_h^n - R_h \bu^n, \quad
  e_{p}^n = p_h^n - R_h p^n, \quad
  e_{\bB}^n = \bB_h^n - R_h \bB^n
\end{align}
for $n=0$, $1$, $\ldots$, $N$.

\subsubsection{The proof of the error estimate in Theorem \ref{thm-main-error}}
\begin{proof}
At the initial time step, we have
\begin{align}
& \| e_{\bu}^{0}\|_{L^2}^2 + \| e_{\bB}^{0}\|_{L^2}^2 \le C h^{2s} \, .
\nn
\end{align}
By the projection \refe{projector-stokes}
and the regularity assumption \refe{regularity-2},
one can verify that the exact solution satisfies the 
formulation below
\begin{empheq}[left=\empheqlbrace]{align} 
  & (D_\tau \bu^n, \bv_h)
  + \frac{1}{R_e} (\nabla R_h \overline{\bu}^n, \nabla \bv_h)
  + \frac{1}{2} \big[ ( \bu^{n-1} \cdot \nabla \, \overline{\bu}^n, \bv_h)
    - (\bu^{n-1} \cdot \nabla \, \bv_h \,,\, \overline{\bu}^n ) \big]
  \nn \\
  & \qquad \quad \quad ~\,~
  - (R_h{p}^n \,,\, \nabla \cdot \bv_h)
  - S_c (\C \, \overline{\bB}^{n} \, \times \bB^{n-1} \, , \, \bv_h)
  = (\mathbf{f}^{n}, \bv_h) + R_1(\bv_h),
  \label{weak-fem-1}\\[5pt]
  & (\nabla \cdot R_h \overline{\bu}^n, q_h) = 0,
  \label{weak-fem-3}\\[5pt]
  & ( D_\tau \bB^n , \bC_h )
  + \frac{S_c} {R_m} \left( \C \, \Pi_h \overline{\bB}^n, \C \, \bC_h \right)
  - S_c \left( \overline{\bu}^n \times \bB^{n-1}, \C \, \bC_h \right)
  = R_2(\bC_h) ,
  \label{weak-fem-2}
\end{empheq}
for any $(\bv_h,q_h,\bC_h) \in \bX_{h}^k \times {M}_{h}^k \times {\bQ}_h^{\widehat{k}}$.
Here the two truncation error terms are defined by
\begin{align}
  & R_1(\bv_h) =
  (D_\tau \bu^n - \frac{\partial \bu}{\partial t}(\bm{x},t^{n}) , \bv_h)
  + \frac{1}{R_e} (\nabla (R_h \overline{\bu}^n - R_h{\bu}^{n}), \nabla \bv_h)
  \nn \\
  & \qquad  \qquad 
  + \Big( 
  \frac{1}{2} \big[ ( \bu^{n-1} \cdot \nabla \, \overline{\bu}^n \,,\, \bv_h)
    - (\bu^{n-1} \cdot \nabla \, \bv_h \,,\, \overline{\bu}^n ) \big]
  - (\bu^{n} \cdot \nabla \, \bu^{n}, \bv_h)\Big)
  \nn \\
  & \qquad \qquad 
  + S_c (\C \, \bB^{n} \, \times \bB^{n}
  - \overline{\bB}^{n} \, \times \bB^{n-1} \, , \, \bv_h)\,,
  \label{truncation-term-1}
  \\[5pt]
  & R_2(\bC_h) =
  (D_\tau \bB^n - \frac{\partial \bB}{\partial t}(\bm{x},t^{n}), \bC_h)
  + S_c R_m^{-1} \left( \C \, (\Pi_h\overline{\bB}^n-{\bB}^{n}), \C \, \bC_h \right)
  \nn \\
  & \qquad  \qquad 
  + S_c (\bu^{n} \times \bB^{n}
  - \overline{\bu}^n \times \bB^{n-1} \,,\, \C \, \bC_h).
\end{align}
Then, subtracting \refe{weak-fem-1}-\refe{weak-fem-2}
from the FEM system  \refe{fem-1}-\refe{fem-3} gives the error equations
\begin{align} 
  & (D_\tau e_\bu^n, \bv_h)
  + R_e^{-1} (\nabla \overline{e}_\bu^n \, , \, \nabla \bv_h)
  - ({e}_p^n, \nabla \cdot \bv_h)
  \nn \\
  & 
  = \frac{1}{2}\left( \big[( \bu^{n-1} \cdot \nabla \, \overline{\bu}^n, \bv_h)
    - (\bu^{n-1} \cdot \nabla \, \bv_h \, ,\, \overline{\bu}^n ) \big]
  - \big[ ( \bu_h^{n-1} \cdot \nabla \, \overline{\bu}_h^n, \bv_h)
    - (\bu_h^{n-1} \cdot \nabla \, \bv_h, \overline{\bu}_h^n ) \big]
  \right)
  \nn \\
  & 
  + \left(S_c (\C \, \overline{\bB}_h^{n} \, \times \bB_h^{n-1} \, , \, \bv_h)
  - S_c (\C \, \overline{\bB}^{n} \, \times \bB^{n-1} \, , \, \overline{\bv}_h)\right)
  -(D_\tau \theta_\bu^n, \bv_h) - R_1(\bv_h),
  \label{error-fem-1}\\[5pt]
  & (\nabla \cdot \overline{e}_\bu^n, q_h) = 0,
  \label{error-fem-3}\\[5pt]
  & ( D_\tau e_\bB^n , \bC_h )
  + S_c R_m^{-1} (\C \, \overline{e}_\bB^n, \C \, \bC_h)
  \nn \\
  & 
  = S_c \left( \overline{\bu}_h^n \times \bB_h^{n-1} - \overline{\bu}^n \times \bB^{n-1},
  \C \, \bC_h \right)
  - ( D_\tau \theta_\bB^n , \bC_h ) - R_2(\bC_h),
  \label{error-fem-2}
\end{align} 
for any $(\bv_h,q_h,\bC_h) \in \bX_{h}^k \times {M}_{h}^k \times {\bQ}_h^{\widehat{k}}$.
We take $\bv_h = \overline{e}_{\bu}^n$ in \refe{error-fem-1},
$q_h = {e}_{p}^n$ in \refe{error-fem-3}
and $\bC_h = \overline{e}_{\bB}^n$ in \refe{error-fem-2}, respectively,
and summing up the results to derive that
\begin{align} 
  & (D_\tau e_\bu^n, \overline{e}_\bu^n) + ( D_\tau e_\bB^n , \overline{e}_\bB^n)
  + R_e^{-1} \|\nabla \, \overline{e}_\bu^n \|_{L^2}^2
  + S_c R_m^{-1} \|\C \, \overline{e}_\bB^n\|_{L^2}^2
  \nn \\
  &
  = -\frac{1}{2}\Big\{
  \big[ ( \bu_h^{n-1} \cdot \nabla \, \overline{\bu}_h^n, \overline{e}_\bu^n)
    - (\bu_h^{n-1} \cdot \nabla \, \overline{e}_\bu^n, \overline{\bu}_h^n ) \big]
  - \big[ ( \bu^{n-1} \cdot \nabla \, \overline{\bu}^n, \overline{e}_\bu^n)
    - (\bu^{n-1} \cdot \nabla \, \overline{e}_\bu^n, \overline{\bu}^n ) \big]\Big\}
  \nn \\
  & 
  \quad + S_c \left(\C \, \overline{\bB}_h^{n} \, \times \bB_h^{n-1}
  - \C \, \overline{\bB}^{n} \, \times \bB^{n-1} \, , \, \overline{e}_\bu^n\right)
  \nn \\
  & 
  \quad + S_c\left( \overline{\bu}_h^n \times \bB_h^{n-1}
  - \overline{\bu}^n \times \bB^{n-1},
  \C \, \overline{e}_\bB^n \right)
  \nn \\
  & 
  \quad -(D_\tau \theta_\bu^n, \overline{e}_\bu^n)
  - ( D_\tau \theta_\bB^n , \overline{e}_\bB^n)
  - R_1(\overline{e}_\bu^n) - R_2(\overline{e}_\bB^n)
  \nn \\
  & := I_1(\overline{e}_\bu^n) + I_2(\overline{e}_\bu^n)
  + I_3(\overline{e}_\bB^n)
  -(D_\tau \theta_\bu^n, \overline{e}_\bu^n) 
  - ( D_\tau \theta_\bB^n , \overline{e}_\bB^n)
  - R_1(\overline{e}_\bu^n) - R_2(\overline{e}_\bB^n).
  \label{error-fem-1-pre}
\end{align} 
By noting the regularity assumption \refe{regularity-1}-\refe{regularity-2}
and Lemma \ref{projection-stokes-mag}, the linear terms on the right hand side
of \refe{error-fem-1-pre} satisfy
\begin{align} 
  & \tau \sum_{m=1}^n \left\{
  - (D_\tau \theta_\bu^m, \overline{e}_\bu^m) 
  - ( D_\tau \theta_\bB^m , \overline{e}_\bB^m)
  - R_1(\overline{e}_\bu^m) - R_2(\overline{e}_\bB^m) \right\}
  \nn \\
  &\le \tau \sum_{m=1}^n \left\{
  \epsilon \|\overline{e}_\bu^m\|_{H^1}^2
  + \epsilon \|\C \, \overline{e}_\bB^m\|_{H^1}^2
  + C \|e_\bu^m\|_{L^2}^2 + C\|e_\bB^m\|_{L^2}^2
  + C\|e_\bu^{n-1}\|_{L^2}^2 + C\|e_\bB^{n-1}\|_{L^2}^2 \right\}
  \nn \\
  & \quad 
  + C\tau^2 + \epsilon^{-1}Ch^{2s}  \,.
  \label{term-linear}
\end{align} 
Next, we estimate the three nonlinear terms one by one.
The first term $I_1(e_\bu^n)$ can be rewritten by
\begin{eqnarray} 
  I_1(\overline{e}_\bu^n) & = &
  - \frac{1}{2}\Big\{
  ( \bu_h^{n-1} \cdot \nabla \, \overline{\bu}_h^n
  - \bu^{n-1} \cdot \nabla \, \overline{\bu}^n \, , \, \overline{e}_\bu^n)
  - \big[(\bu_h^{n-1} \cdot \nabla \, \overline{e}_\bu^n \, ,\, \overline{\bu}_h^n )
    - (\bu^{n-1} \cdot \nabla \, \overline{e}_\bu^n \, , \, \overline{\bu}^n )\big]
  \Big\}
  \nn \\
  & = & -\frac{1}{2} \left( (e_\bu^{n-1}+\theta_\bu^{n-1}) \cdot \nabla \,
  \overline{\bu}_h^n \, , \,\overline{e}_\bu^n\right)
  - \frac{1}{2} \big( \bu^{n-1}\cdot \nabla \,
  (\overline{e}_\bu^n + \overline{\theta}_\bu^n) \, ,\, \overline{e}_\bu^n \big)
  \nn \\
  &  & + \frac{1}{2}\big(\bu^{n-1} \cdot \nabla \, \overline{e}_\bu^n \, , \,
  \overline{e}_\bu^n + \overline{\theta}_\bu^n \big)
  + \frac{1}{2}\left((e_\bu^{n-1}+\theta_\bu^{n-1})\cdot \nabla \,
  \overline{e}_\bu^n \, , \, \overline{\bu}_h^n \right)
  \nn \\
  & := & \sum_{i=1}^4 J_i^n \,.
  \label{term-1}
\end{eqnarray}
We now estimate $\{J_i^n\}_{i=1}^4$.
The term $J_1^n + J_4^n$ can be bounded by
\begin{eqnarray}
  J_1^n + J_4^n & = &
  -\frac{1}{2} \left( (e_\bu^{n-1}+\theta_\bu^{n-1}) \cdot \nabla \,
  \overline{e}_{\bu}^n \, , \,\overline{e}_\bu^n\right)
  - \frac{1}{2} \left( (e_\bu^{n-1}+\theta_\bu^{n-1}) \cdot \nabla \,
  R_h\overline{\bu}^n,  \overline{e}_\bu^n\right)
  \nn \\
  &  & + \frac{1}{2}\left((e_\bu^{n-1}+\theta_\bu^{n-1})\cdot \nabla \,
  \overline{e}_\bu^n \, , \, \overline{e}_{\bu}^n \right)
  + \frac{1}{2}\left((e_\bu^{n-1}+\theta_\bu^{n-1})\cdot \nabla \,
  \overline{e}_\bu^n \, , \, R_h\overline{\bu}^n \right)
  \nn \\
  & = & - \frac{1}{2} \left( (e_\bu^{n-1}+\theta_\bu^{n-1}) \cdot \nabla \,
  R_h\overline{\bu}^n, \overline{e}_\bu^n\right)
  + \frac{1}{2}\left((e_\bu^{n-1}+\theta_\bu^{n-1})\cdot \nabla \,
  \overline{e}_\bu^n \, , \, R_h\overline{\bu}^n \right)
  \nn \\
  & \le &
  C \|e_\bu^{n-1}+\theta_\bu^{n-1}\|_{L^{2}}\|R_h\overline{\bu}^n\|_{W^{1,3}}
  \|\overline{e}_\bu^n\|_{L^{6}}
  +C \|e_\bu^{n-1}+\theta_\bu^{n-1}\|_{L^{2}}\|\overline{e}_\bu^n\|_{H^{1}}
  \|R_h\overline{\bu}^n\|_{L^{\infty}}
  \nn \\
  & \le &  C \|e_\bu^{n-1}+\theta_\bu^{n-1}\|_{L^{2}}\|\overline{e}_\bu^n\|_{H^{1}}
  \nn \\
  & \le & \epsilon \|\overline{e}_\bu^n\|_{H^{1}}^2
  + \epsilon^{-1} C \|e_\bu^{n-1}\|_{L^{2}}^2
  + \epsilon^{-1} C h^{2s} \,,
  \label{term-j-1-4}
\end{eqnarray}
where we have used the uniform boundedness results for $R_h\overline{\bu}^n$
in \refe{uniform-bound-u}.
And $J_2^n+J_3^n$ can be estimated directly
\begin{eqnarray}
  J_2^n+J_3^n & \le &
  C\|\bu^{n-1}\|_{L^{\infty}}
  \|\overline{\theta}_\bu^n\|_{H^1}
  \|\overline{e}_\bu^n \|_{L^2}
  + C \|\bu^{n-1}\|_{L^{\infty}}
  \|\overline{e}_\bu^n\|_{H^{1}} \|\overline{\theta}_\bu^n\|_{L^{2}}
  \nn \\
  & \le &
  C h^s \|\overline{e}_\bu^n \|_{L^2}
  + C \|\overline{e}_\bu^n\|_{H^{1}}  h^s
  \nn \\
  & \le & \epsilon \|\overline{e}_\bu^n \|_{H^1}^2
  + \epsilon^{-1}C (\|e_\bu^n \|_{L^2}^2 + \|e_\bu^{n-1} \|_{L^2}^2 + h^{2s}) \,.
  \label{term-j-2-3}
\end{eqnarray}
With the above estimates \refe{term-j-1-4} and \refe{term-j-2-3},
we get the following estimate
\begin{align} 
  I_1(\overline{e}_\bu^n) \le \epsilon \|\overline{e}_\bu^n\|_{H^{1}}^2
  + \epsilon^{-1} C (\|e_\bu^{n}\|_{L^{2}}^2 +\|e_\bu^{n-1}\|_{L^{2}}^2 + h^{2s}) \,.
  \label{term-1-end}
\end{align}
Then, we turn to estimate $I_2(\overline{e}_\bu^n)$ and $I_3(\overline{e}_\bB^n)$,
which can be rewritten by
\begin{eqnarray}
  I_2(\overline{e}_\bu^n) 
  & = &
  S_c \left( \C \, \overline{e}_\bB^{n} \, \times \bB_h^{n-1}\, , \, \overline{e}_\bu^n \right)
  + S_c\left( \C \, \Pi_h \overline{\bB}^{n} \, \times \bB_h^{n-1}
  - \C \, \overline{\bB}^{n} \, \times \bB^{n-1} \, , \, \overline{e}_\bu^n \right)
  \nn \\
  & = &
  S_c \left( \C \, \overline{e}_\bB^{n} \, \times \bB_h^{n-1}\, , \, \overline{e}_\bu^n \right)
  + S_c \left( \C \, \Pi_h \overline{\bB}^{n} \, \times (e_\bB^{n-1}+\theta_\bB^{n-1}),
  \overline{e}_\bu^n \right)
  \nn \\
  &  & + S_c\left( \C \, \overline{\theta}_\bB^{n} \,\times \bB^{n-1} \,,\,
  \overline{e}_\bu^n \right)\,,
  \nn \\[7pt]
  I_3(\overline{e}_\bB^n) 
  & = &
  S_c \left(\overline{e}_\bu^n \times \bB_h^{n-1}, \C \, \overline{e}_\bB^n\right)
  + S_c\left(R_h\overline{\bu}^n \times \bB_h^{n-1} - \overline{\bu}^n \times \bB^{n-1},
  \C \, \overline{e}_\bB^n\right)
  \nn \\
  & = &
  S_c \left(\overline{e}_\bu^n \times \bB_h^{n-1}, \C \, \overline{e}_\bB^n\right)
  + S_c\left(R_h\overline{\bu}^n \times (e_\bB^{n-1}+\theta_\bB^{n-1}) \, ,\,
  \C \, \overline{e}_\bB^n\right)
  \nn \\
  & &
  + S_c\left(\overline{\theta}_\bu^n \times \bB^{n-1} \,,\,
  \C \, \overline{e}_\bB^n\right).
  \nn
\end{eqnarray}
By noting the fact that
\[
S_c \left( \C \, \overline{e}_\bB^{n} \, \times \bB_h^{n-1}\, , \,
\overline{e}_\bu^n \right)
+ S_c \left(\overline{e}_\bu^n \times \bB_h^{n-1}, \C \, \overline{e}_\bB^n\right)
= 0 \,,
\]
we have
\begin{align}
  & I_2(\overline{e}_\bu^n) + I_3(\overline{e}_\bB^n)
  \nn \\
  & = 
  S_c \left( \C \, \Pi_h \overline{\bB}^{n} \, \times (e_\bB^{n-1}+\theta_\bB^{n-1}),
  \overline{e}_\bu^n \right)
  + S_c\big( \C \, \overline{\theta}_\bB^{n} \,\times \bB^{n-1} \,,\,
  \overline{e}_\bu^n \big)
  \nn \\
  & \quad
  + S_c\left(R_h\overline{\bu}^n \times (e_\bB^{n-1}+\theta_\bB^{n-1}) \, ,\,
  \C \, \overline{e}_\bB^n\right)
  + S_c\big(\overline{\theta}_\bu^n \times \bB^{n-1} \,,\,
  \C \, \overline{e}_\bB^n\big)
  \nn \\
  & \le 
  C \|\C \, \Pi_h \overline{\bB}^{n}\|_{L^3} \|e_\bB^{n-1}+\theta_\bB^{n-1}\|_{L^2}
  \|\overline{e}_\bu^n\|_{L^6}
  + C \|\C \, \overline{\theta}_\bB^{n}\|_{L^2} \|\bB^{n-1}\|_{L^3}
  \|\overline{e}_\bu^n\|_{L^6}
  \nn \\
  & \quad
  + C\|R_h\overline{\bu}^n\|_{L^{\infty}} \|e_\bB^{n-1}+\theta_\bB^{n-1}\|_{L^2}
  \|\C \, \overline{e}_\bB^n\|_{L^2}
  + C\|\overline{\theta}_\bu^n\|_{L^6} \|\bB^{n-1}\|_{L^3}
  \|\C \, \overline{e}_\bB^n\|_{L^2}
  \nn \\
  & \le  C (\|e_\bB^{n-1}\|_{L^2} + h^s)\|\overline{e}_\bu^n\|_{H^1}
  + C(\|e_\bB^{n-1}\|_{L^2} + h^s)
  \|\C \, \overline{e}_\bB^n\|_{L^2}
  + C h^s \|\C \, \overline{e}_\bB^n\|_{L^2}
  \nn \\
  & \le 
  \epsilon \|\overline{e}_\bu^n\|_{H^1}^2
  + \epsilon \|\C \, \overline{e}_\bB^n\|_{L^2}^2
  +\epsilon^{-1}C \|e_\bB^{n-1}\|_{L^2}^2
  +\epsilon^{-1}C h^{2s} \,,
  \label{nonlinear-result}
\end{align}
where we have used the uniform boundedness of $\Pi_h \bB$ in \refe{uniform-bound-B}.
Finally, taking estimates \refe{term-1-end} and \refe{nonlinear-result}
into \refe{error-fem-1-pre}, we arrive at
\begin{align} 
  & \frac{\|e_\bu^n\|_{L^2}^2 - \|e_\bu^{n-1}\|_{L^2}^2}{2 \tau}
  + \frac{\|e_\bB^n\|_{L^2}^2 - \|e_\bB^{n-1}\|_{L^2}^2}{2 \tau}
  + \frac{1}{R_e} \|\nabla \overline{e}_\bu^n \|_{L^2}^2
  + \frac{S_c}{R_m} \|\C \, \overline{e}_\bB^n\|_{L^2}^2
  \nn \\[6pt]
  & \le
  \epsilon \|\overline{e}_\bu^n\|_{H^1}^2
  + \epsilon \|\C \, \overline{e}_\bB^n\|_{L^2}^2
  + \epsilon^{-1}C\|e_\bu^n\|_{L^2}^2 
  + \epsilon^{-1}C \|e_\bu^{n-1}\|_{L^2}^2
  + C\|e_\bB^{n-1}\|_{L^2}^2
  + \epsilon^{-1}C\tau^2
  \nn \\
  & \quad
  + \epsilon^{-1}Ch^{2s} 
  - (D_\tau \theta_\bu^n , \overline{e}_\bu^n) 
  - ( D_\tau \theta_\bB^n , \overline{e}_\bB^n)
  - R_1(\overline{e}_\bu^n) - R_2(\overline{e}_\bB^n) \,.
\end{align} 
Then, we chose a small $\epsilon$ and sum up the last inequality
for the index $n=0$, $1$, $\ldots$, $k$ to derive that
\begin{align} 
  & \|e_\bu^n\|_{L^2}^2 + \|e_\bB^n\|_{L^2}^2
  + \tau \sum_{m=1}^n \left( \| \overline{e}_\bu^m \|_{H^1}^2
  + \|\C \, \overline{e}_\bB^m \|_{L^2}^2 \right)
  \nn \\
  & \le \tau \sum_{m=0}^n \left( \|e_\bu^m\|_{L^2}^2
  + C\|e_\bB^m\|_{L^2}^2 \right)
  + C\tau^2 + Ch^{2s} 
\end{align}
where we have used the estimates \refe{term-linear} for the linear terms.
By the discrete Gronwall's inequality in Lemma \ref{gronwall},
when $C \tau \le \frac{1}{2}$, we have
\begin{align}
  & \|e_\bu^n\|_{L^2}^2 + \|e_\bB^n\|_{L^2}^2
  + \tau \sum_{m=1}^n \left( \| \overline{e}_\bu^m \|_{H^1}^2
  + \|\C \, \overline{e}_\bB^m \|_{L^2}^2 \right)
  \nn \\
  & \le
  C\exp(\frac{TC}{ 1- C\tau})  (\tau^2 + h^{2s+2})
  \nn \\
  & \le C\exp(2{TC})  (\tau^2 + h^{2s+2}) \,.
  \label{result-error}
\end{align}
Theorem \ref{thm-main-error} is proved by combining \refe{result-error}
and the projection error estimates in Lemma \ref{projection-stokes-mag}.
\end{proof}

\section{Numerical results}
\label{numer-section}
\setcounter{equation}{0}

In this section, we provide some numerical experiments to confirm
our theoretical analyses and demonstrate the
accuracy, stability and robustness of the proposed linearized
conservative FEM. The computations are performed with
FEniCS \cite{fenics}.

\subsection{Two-dimensional numerical results}

We introduce several two-dimensional operators first.
For scalar function $p$ and vector function $\bB = [B_1,B_2]^T$,
the two-dimensional operators $\D$, $\nabla$, $\curl$ and $\C$ 
are defined by
\begin{align}
  &\nabla \cdot \bB = \frac{\partial B_1}{\partial x}
  + \frac{\partial B_2}{\partial y} , \, 
  \nabla p  = 
  \left[\frac{\partial p}{\partial x}\, ,
    \frac{\partial p}{\partial y} \right]^{T} , \,
  \curl \, \bB = \frac{\partial B_2}{\partial x}
  - \frac{\partial B_1}{\partial y} \, , 
  \C \, p = \left[\frac{\partial p}{\partial y}\, ,
    -\frac{\partial p}{\partial x} \right]^{T}  .
  \nn
\end{align}
In two dimensional space, the dynamical incompressible MHD equations can be reduced to
\begin{empheq}[left=\empheqlbrace]{align} 
  & \frac{\partial \bu}{\partial t} - \frac{1}{R_e} \Delta \bu
  + \bu \cdot \nabla \, \bu
  + \nabla p - S_c \, \curl \bB \,
  \left[
    \begin{array}{r}
       \!\!\! -B_2 \!\!\!
      \\
       \!\!\! B_1 \!\!\!
    \end{array}      
    \right]
  = \mathbf{f},
  && \bm{x} \in \Omega,
  \label{pde-1-2d}\\[4pt]
  & \frac{\partial \bB}{\partial t}
  + \frac{S_c}{R_m} \C \, (\curl \, \bB) - S_c \C \, (u_1 B_2 - u_2 B_1)
  = \mathbf{0},
  && \bm{x} \in \Omega,
  \label{pde-2-2d}\\[4pt]
  & \nabla \cdot \bu=0,
  && \bm{x} \in \Omega,
  \label{pde-3-2d}\\[4pt]
  & \nabla \cdot \bB=0,
  && \bm{x} \in \Omega,
  \label{pde-4-2d}
\end{empheq}
where $\bu=[u_1,u_2]^T$ and $\bB=[B_1,B_2]^T$.
The above equation system is supplemented with homogeneous boundary conditions.
\begin{empheq}[left=\empheqlbrace]{align} 
  & \bu = 0,
  \quad \bm{x} \in \partial \Omega,
  \label{bc-u-2d} \\
  & \bB \cdot \n = 0, \quad \curl \,\bB = 0,
  \quad \bm{x} \in \partial \Omega,
  \label{bc-mag-2d}
\end{empheq}
and initial conditions
\begin{align}
  \bu(\bm{x},0) = \bu_0(\bm{x}),\quad
  \bB(\bm{x},0) = \bB_0(\bm{x}),
  \label{initial-2d}
\end{align}
where $\D \, \bu_0(\bm{x}) = \D \, \bB_0(\bm{x}) = 0$.
It should be remarked that the two-dimensional incompressible MHD equations can be
reformulated into scalar form with vorticity and magnetic stream functions,
i.e., the $\omega-\psi$ formulation, see \cite{CorMar}.
However, for consistency with the original three-dimensional model,
we still use \refe{pde-1-2d}-\refe{initial-2d}.
Analogous to the three-dimensional scheme, 
the linearized backward Euler FEM for \refe{pde-1-2d}-\refe{initial-2d}
is to look for
$(\bu_h^n, \bB_h^n, p_h^n) \in \bX_h^k \times \bQ_h^k \times M_h^k$
with $(p_h^{n},1)=0$,
such that for any $(\bv_h,\bC_h,q_h)\in \bX_{h}^k \times {\bQ}_h^k \times {M}_{h}^k$
\begin{empheq}[left=\empheqlbrace]{align} 
  & (D_\tau \bu_h^n, \bv_h)
  + \frac{1}{R_e} (\nabla \overline{\bu}_h^n, \nabla \bv_h)
  + \frac{1}{2} \big[ ( \bu_h^{n-1} \cdot \nabla \, \overline{\bu}_h^n, \bv_h)
    - (\bu_h^{n-1} \cdot \nabla \, \bv_h, \overline{\bu}_h^n ) \big]
  \nn \\
  & \qquad  \qquad  \quad 
  - ({p}_h^n, \nabla \cdot \bv_h)
  + S_c \left(\curl \, \overline{\bB}_h^{n} \, , \,
  (\bB_h^{n-1})^T \mathsf{T}_{\pi/2} \bv_h\right)
  = (\mathbf{f}^{n}, \bv_h)
  \label{fem-1-2d}\\[5pt]
  & ( D_\tau \bB_h^n , \bC_h )
  + \frac{S_c}{R_m} (\curl \, \overline{\bB}_h^n, \curl \, \bC_h)
  - S_c \left((\bB_h^{n-1})^T \mathsf{T}_{\pi/2} \overline{\bu}_h^n\, ,\,
  \curl \, \bC_h\right)
  = 0\,,
  \label{fem-2-2d}\\[5pt]
  & (\nabla \cdot \overline{\bu}_h^n, q_h) = 0,
  \label{fem-3-2d}
\end{empheq}
where $\mathsf{T}_{\pi/2}$ is the rotation matrix defined by
\[ \mathsf{T}_{\pi/2} = 
\left[\begin{array}{rr}
    0  & -1\\
    1  & 0
  \end{array}
  \right] \,.
\]
Again, we take
$\bu_h^0 = \Pi_h \bu^0$ and $\bB_h^0 = \Pi_h \bB^0$
at the initial time step.

\begin{example}
  \label{example-2d-artificial}
We test the performance of the proposed scheme \refe{fem-1-2d}-\refe{fem-3-2d}
for a two-dimensional Hartmann flow problem, see \cite{HuMaXu}.
We set $\Omega = (0,1)^2$ and $R_e=S_c=R_m=1.0$ in this example.
The Hartmann flow problem has an explicit analytic solution
as follows
\begin{align}
  \bu = (u_1, 0)^T,
  \quad
  p = -x-\frac{(B_1)^2}{2},
  \quad
  \bB = (B_1, 1)^T,
  \label{exact-example-1}
\end{align}
where
\[
u_1 = \frac{\cosh(\frac{1}{2}) -\cosh(y)}{2\sinh(\frac{1}{2})},
\quad
B_1 = \frac{\sinh(y) - 2\sinh(\frac{1}{2}) \, y}{2\sinh(\frac{1}{2})} \,.
\]
It should be noted that the expressions \refe{exact-example-1}
satisfy the following stationary incompressible MHD equations
\begin{empheq}[left=\empheqlbrace]{align} 
  & - \Delta \bu
  + \bu \cdot \nabla \, \bu
  + \nabla p - \curl \bB \,
  \left[
    \begin{array}{r}
       \!\!\! -B_2 \!\!\!
      \\
       \!\!\! B_1 \!\!\!
    \end{array}      
    \right]
  = \mathbf{0},
  && \bm{x} \in \Omega,
  \label{pde-1-2d-stationary}\\
  &\C \, (\curl \, \bB) - \C \, (u_1 B_2 - u_2 B_1)
  = \mathbf{0},
  && \bm{x} \in \Omega,
  \label{pde-2-2d-stationary}\\
  & \nabla \cdot \bu=0,
  && \bm{x} \in \Omega,
  \label{pde-3-2d-stationary}\\
  & \nabla \cdot \bB=0,
  && \bm{x} \in \Omega.
  \label{pde-4-2d-stationary}
\end{empheq}
However, the exact solution \refe{exact-example-1}
does not fulfill the boundary condition
\refe{bc-u-2d}-\refe{bc-mag-2d} and constraint $(p,1)=0$.
Therefore, we simply take the Dirichlet boundary condition
based on the analytic solutions for $\bu_h^n$ and $\bB_h^n$,
while we set $p_h^n = p(0,0)$ on the node at the Origin.
For the initial solution, we take
\begin{align}
  \bu = (1,0)^T, \quad
  \bB = (0,1)^T.
\end{align}
In the computation, we set $\tau = 0.005$ and use
$\bX_h^2 \times \bQ_h^2 \times M_h^1$
on a uniform triangular mesh generated by FEniCS
with $h=\frac{1}{128}$.

The numerical results obtained at $T=10$ are shown
in Figure \ref{hartmann-2d}.
It is easy to see that the numerical results computed by
the proposed linearized mixed FEM agree well with \cite{HuMaXu},
where a divergence free scheme was used to solve the incompressible MHD equations.
We also plot The error functions in Figure \ref{hartmann-errror-2d}.
\begin{figure}[ht]
\centering
\begin{tabular}{c}
\epsfig{file=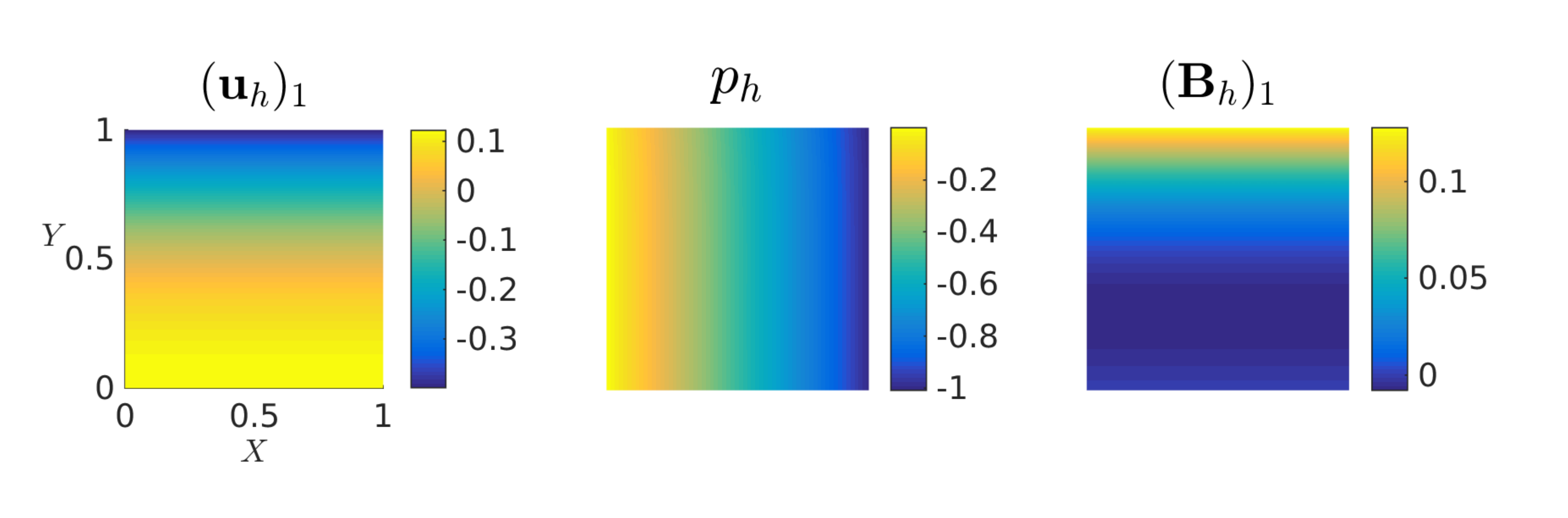,height=2.0in,width=6in}
\end{tabular}
\caption{Numerical results of $(\bu_h)_1$, $p_h$ and $(\bB_h)_1$
  at $T=10$ computed by \refe{fem-1-2d}-\refe{fem-3-2d}
  with $\tau=0.005$ and $h=1/128$. (Example \ref{example-2d-artificial})}
\label{hartmann-2d}
\end{figure}
\begin{figure}[ht]
  \centering
  \begin{tabular}{c}
    \epsfig{file=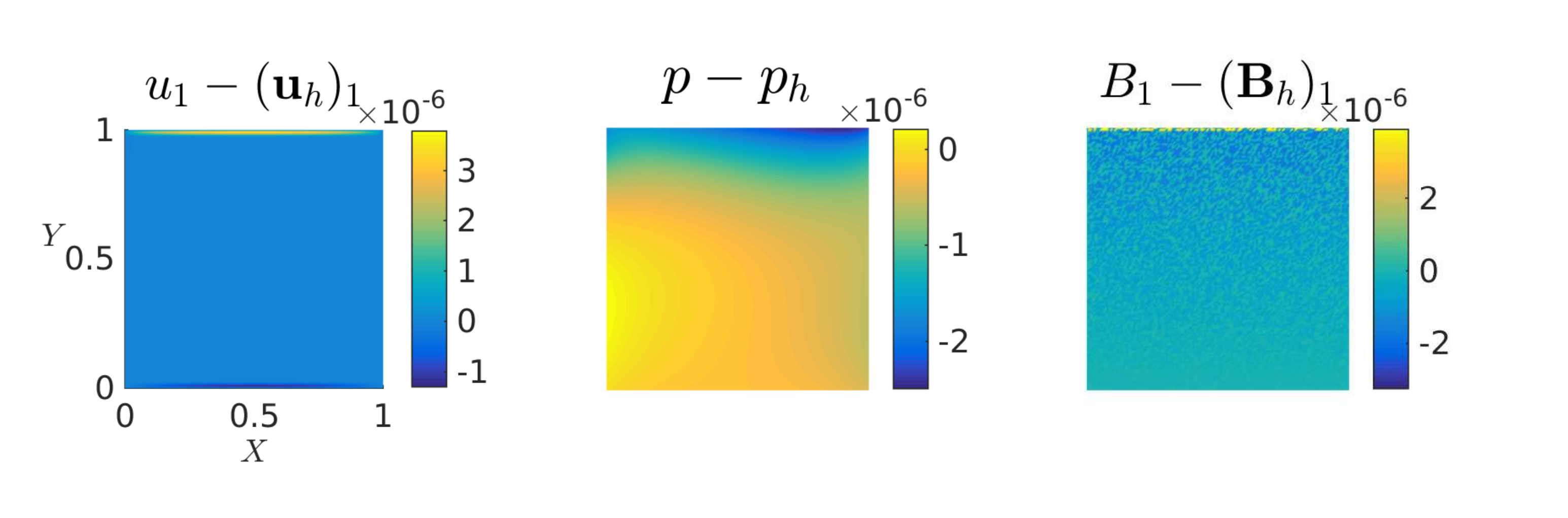,height=2.0in,width=6in}
  \end{tabular}
  \caption{Error plots of $(\bu_h)_1$, $p_h$ and $(\bB_h)_1$
    at $T=10$ by \refe{fem-1-2d}-\refe{fem-3-2d}
    with $\tau=0.005$ and $h=1/128$. (Example \ref{example-2d-artificial})}
  \label{hartmann-errror-2d}
\end{figure}

\end{example}

\subsection{Three-dimensional numerical experiments}

\begin{example}
  \label{example-3d-artificial}
In this example, we test the convergence of the proposed scheme
\refe{fem-1}-\refe{fem-3} 
for a three-dimensional artificial problem. Here $\Omega= (0,1)^3$
and $R_e=S_c=R_m=1.0$. The exact solution is taken to be
\begin{align}
  \bu = \left[
    \begin{array}{r}
      e^{t} \cos(y)
      \\
      e^{t} \cos(z)
      \\
      e^{t} \cos(x)      
    \end{array}      
    \right],
  \quad
  p = e^{t} (x-0.5) \cos(y) \sin(z),
  \quad
  \bB = \left[
    \begin{array}{r}
      e^{t} \sin(y)
      \\
      e^{t} \sin(z)
      \\
      e^{t} \cos(x)      
    \end{array}      
    \right].
  \label{exact-example-3d}
\end{align}
Here, one can verify that $(p,1) = 0$.
A uniform mesh is used in our computation.
In each direction, there are $M+1$ vertices and therefore $h = \sqrt{3}/M$,
see Figure \ref{mesh-3d-4x4} for illustration when $M=4$.
\begin{figure}[ht]
  \centering
  \begin{tabular}{c}
    \epsfig{file=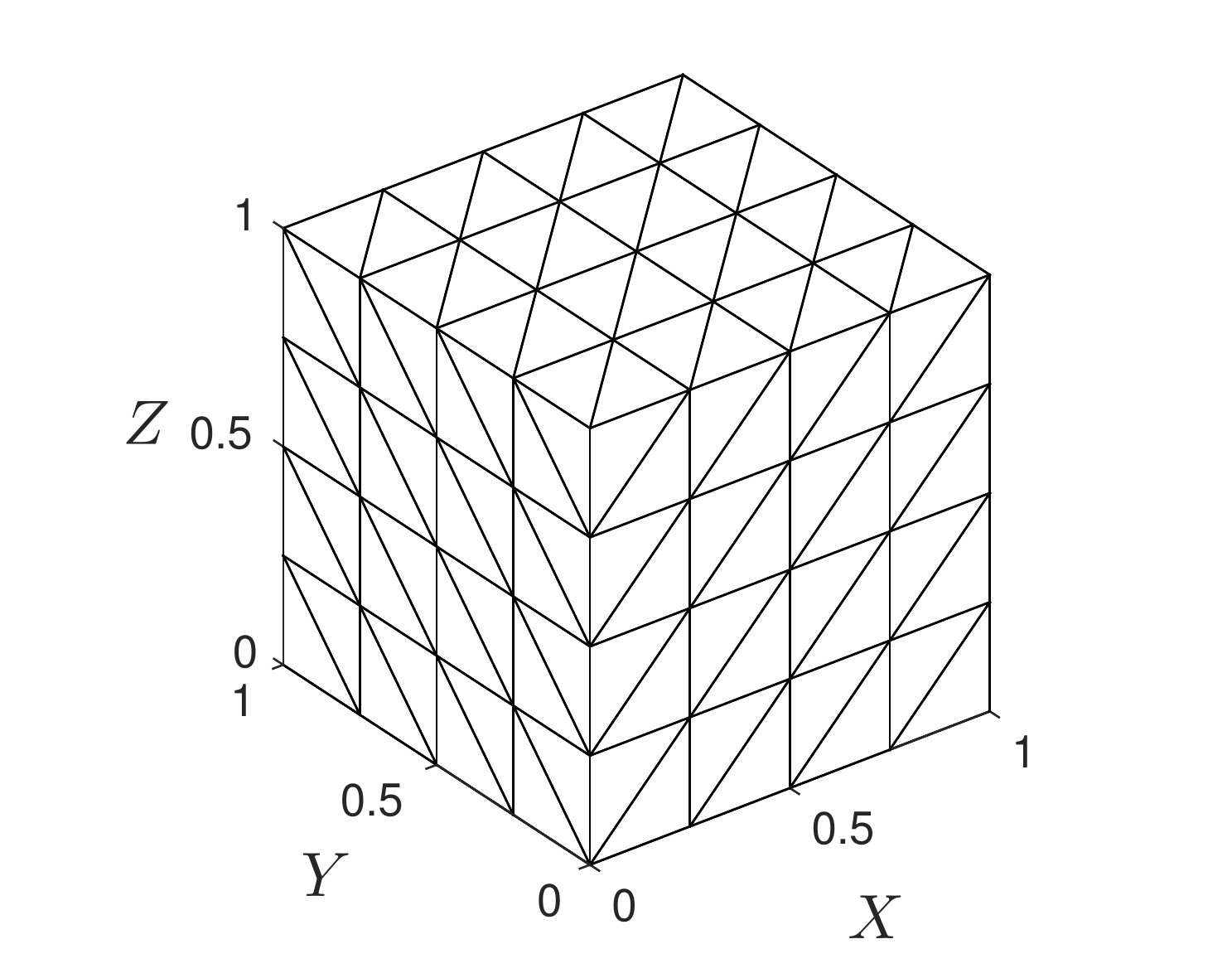,height=2.4in,width=3in}
  \end{tabular}
  \caption{An illustration of uniform tetrahedral mesh with $M=4$.
    (Example \ref{example-3d-artificial})}
  \label{mesh-3d-4x4}
\end{figure}

To show the optimal convergence rate of the proposed method,
we take $\tau = \frac{1}{2M}$ for the lowest order method
$\bX_h^2 \times \bQ_h^1 \times M_h^1$
and $\tau = \frac{1}{M^2}$ for $\bX_h^2 \times \bQ_h^2 \times M_h^1$, respectively.
The $L^2$ errors are shown in Table \ref{error-3d},
where first order convergence for $\bX_h^2 \times \bQ_h^1 \times M_h^1$
and second order convergence for $\bX_h^2 \times \bQ_h^2 \times M_h^1$
are obtained.
Numerical results from Table \ref{error-3d} verified that
the proposed linearized scheme has optimal convergence rate,
provided that the exact solution is smooth enough.

\begin{table}[ht]
\begin{center}
\caption{ $L^2$-norm errors for the linearized scheme on the unit cube.
(Example \ref{example-3d-artificial})} \label{error-3d}
\begin{tabular}{|c|c|c|c|c|c|c|}
\hline
\multicolumn{7}{|c|}{
$\bX_h^2 \times \bQ_h^1 \times M_h^1$
  \quad ($\tau = \frac{1}{2M}$)} \\
\hline
&  $\|\bu_h^{N} - \bu\|_{L^2}$
& Order
&  $\|p_h^{N} - p\|_{L^2}$
& Order
&  $\|\bB_h^{N} - \bB\|_{L^2}$
& Order
\\
\hline
M= 4  & 6.4876e-03  & $-$     & 4.4078e-01 & $-$    & 2.6777e-01 & $-$    \\
M= 8  & 3.3178e-03  & 0.9674  & 2.2073e-01 & 0.9978 & 1.3366e-01 & 1.0024 \\
M= 16 & 1.6613e-03  & 0.9979  & 1.1004e-01 & 1.0043 & 6.6819e-02 & 1.0002 \\
\hline
\hline
\multicolumn{7}{|c|}{$\bX_h^2 \times \bQ_h^2 \times M_h^1$
  \quad ($\tau = \frac{1}{M^2}$)} \\
\hline
&  $\|\bu_h^{N} - \bu\|_{L^2}$
& Order
&  $\|p_h^{N} - p\|_{L^2}$
& Order
&  $\|\bB_h^{N} - \bB\|_{L^2}$
& Order
\\
\hline
M= 4  & 3.2879e-03  & $-$    & 2.0861e-01 & $-$    & 4.2069e-02 & $-$   \\
M= 8  & 8.3121e-04  & 1.9839 & 5.3957e-02 & 1.9509 & 1.0361e-02 & 2.0216\\
M= 16 & 2.0776e-04  & 2.0003 & 1.3611e-02 & 1.9870 & 2.5797e-03 & 2.0059\\
\hline
\end{tabular}
\end{center}
\end{table}

\end{example}

\begin{example}
  \label{example-3d-cavity-flow}

\begin{figure}[ht]
  \centering
  \begin{tabular}{cc}
    \epsfig{file=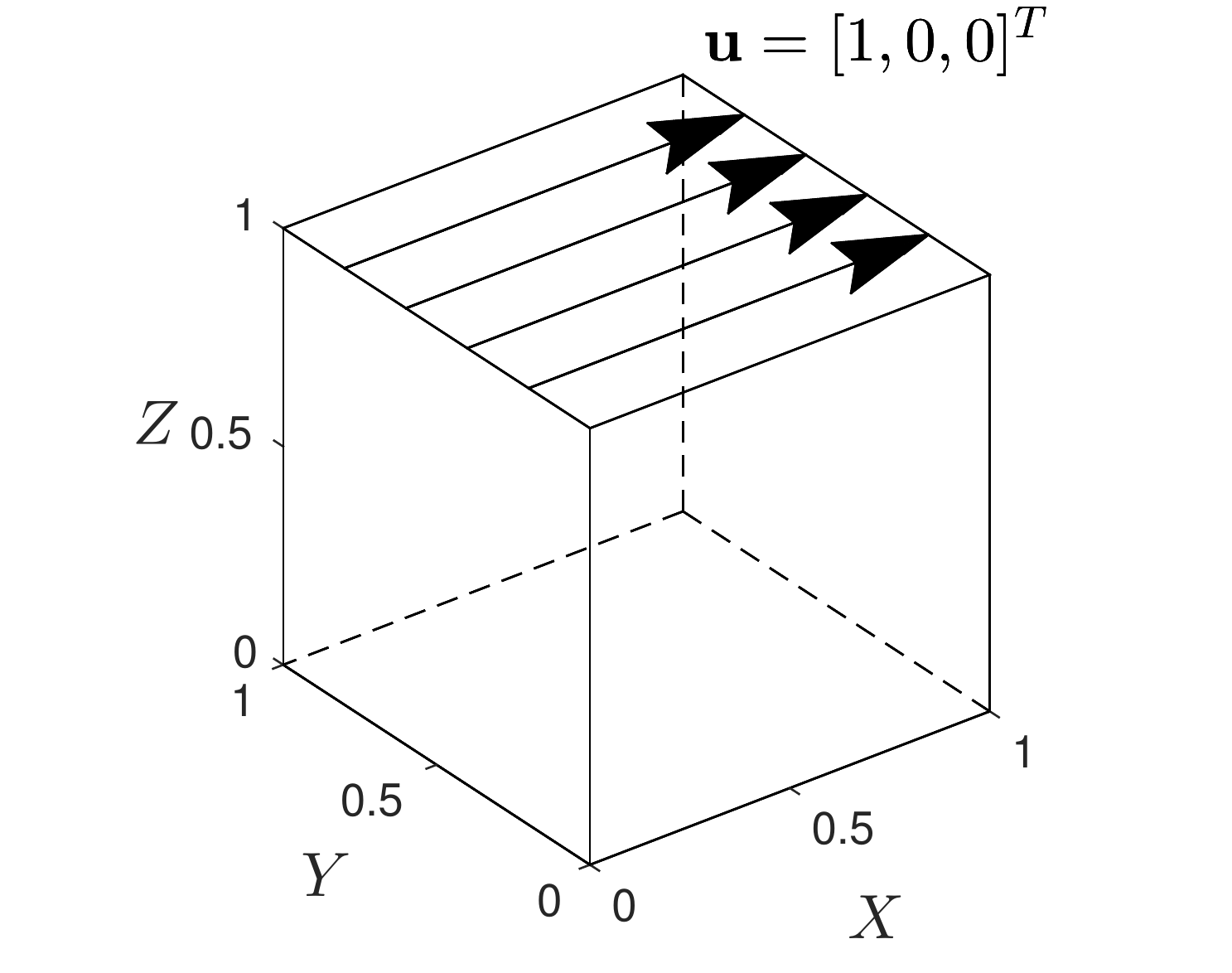,height=2.5in,width=3in}&
    \epsfig{file=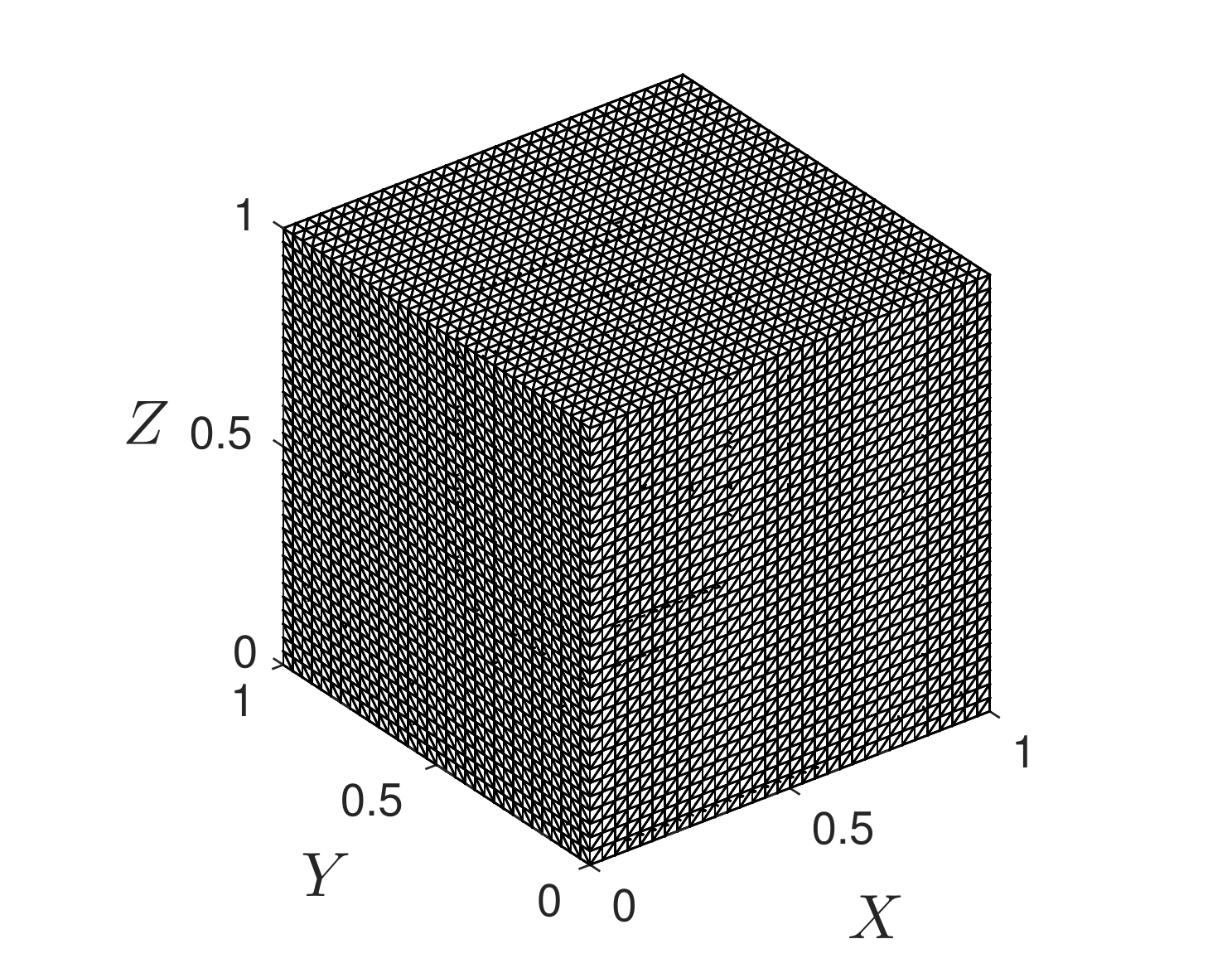,height=2.5in,width=3in}    
  \end{tabular}
  \caption{Left: illustration of the problem.
    Right: the mesh used in the computation, $M=\frac{1}{32}$
    $35,937$ nodes, $196,608$ elements.
    $238,688$ edges and $399,360$ faces.
    (Example \ref{example-3d-cavity-flow})}
  \label{flow_illustration}
\end{figure}
  
In the final example, we test the performance of the proposed scheme
for a benchmark lid driven cavity problem on the unit cube domain $\Omega= (0,1)^3$
see Figure \ref{flow_illustration}.
Here, the physical parameters are chosen to satisfy
the following incompressible MHD equations
\begin{empheq}[left=\empheqlbrace]{align} 
  & \frac{\partial \bu}{\partial t} - \frac{1}{100} \Delta \bu
  + \bu \cdot \nabla \, \bu
  + \nabla p - \frac{1}{20} \, \C \, \bB \, \times\bB =\mathbf{f},
  && \bm{x} \in \Omega,
  \label{lid-cavity-1}\\[4pt]
  & \frac{\partial \bB}{\partial t}
  + \frac{1}{200} \C \, (\C \, \bB) - \C \, (\bu \times \bB)
  = \mathbf{0},
  && \bm{x} \in \Omega,
  \label{lid-cavity-2}\\[4pt]
  & \nabla \cdot \bu=0,
  && \bm{x} \in \Omega,
  \label{lid-cavity-3}\\[4pt]
  & \nabla \cdot \bB=0,
  && \bm{x} \in \Omega,
  \label{lid-cavity-4}
\end{empheq}
The initial conditions are taken to be
\begin{align}
  \bu_0 = (v,0,0)^T, \quad
  \bB_0 = (1,0,0)^T
\end{align}
where
\begin{align}
  v(x,y,z) = 
  \left\{
    \begin{array}{ll}
      \left(1+\frac{1}{\alpha}(z-1) \right)^2, &  z \ge 1-\alpha , \quad
      \alpha = 0.001, \\
      0, & \textrm{otherwise}
    \end{array}
    \right.  
\end{align}    
The boundary condition on $\partial \Omega$ are set to be
\begin{align}
  \bu = \bu_0, \quad  \bB \cdot \mathbf{n} = 0, \quad
  \C \, \bB \times \mathbf{n} = \mathbf{0}.
\end{align}
This example was tested in \cite{Hiptmair_LMZ},
where a divergence free approach was used.
Some similar and simplified two-dimensional model
were tested by several authors with different methods, e.g.,
see \cite{Marioni,Phillips,Shatrov}.
It should be noted that due to the different nondimensionalization procedure,
the unknown $\bB$ in the above MHD equations
\refe{lid-cavity-1}-\refe{lid-cavity-4}
is a little bit different with the one used in \refe{pde-1}-\refe{pde-4}.
However, the scheme can be applied to 
\refe{lid-cavity-1}-\refe{lid-cavity-4} with a slight modification
on the coefficients.

In the computation, we use $\bX_h^2 \times \bQ_h^2 \times M_h^1$
on a uniform tetrahedral mesh with $M=32$.
There are $823,875$ dofs for $\bu_h$,
$35,937$ dofs for $p_h$ and $1,276,096$ dofs for $\bB_h$, respectively.
The time step $\tau = 0.01$  is used in the computation.

As numerical results reported in \cite[Example 5.4]{Hiptmair_LMZ},
shows that the stationary state arrives at $T=4$,
we show the numerical results at $T=4$ in Figures \ref{streamline} and 
streamline of $\bu_h$ and at $y=0.5$ in Figure \ref{streamline}.
Numerical experiments with finer time step $\tau = 0.005$
have been done to verify the streamline pattern.
From Figures \ref{streamline}, we see that
the streamline at $y=0.5$ obtained by the proposed scheme \refe{fem-1}-\refe{fem-3}
is very similar to those reported in \cite[Example 5.4]{Hiptmair_LMZ},
where a divergence free scheme was used.
The pressure contour plots in Figure \ref{pressurecontour} also agree well with
previous results in \cite[Example 5.4]{Hiptmair_LMZ}.
\begin{figure}[ht]
  \centering
  \begin{tabular}{c}
    \epsfig{file=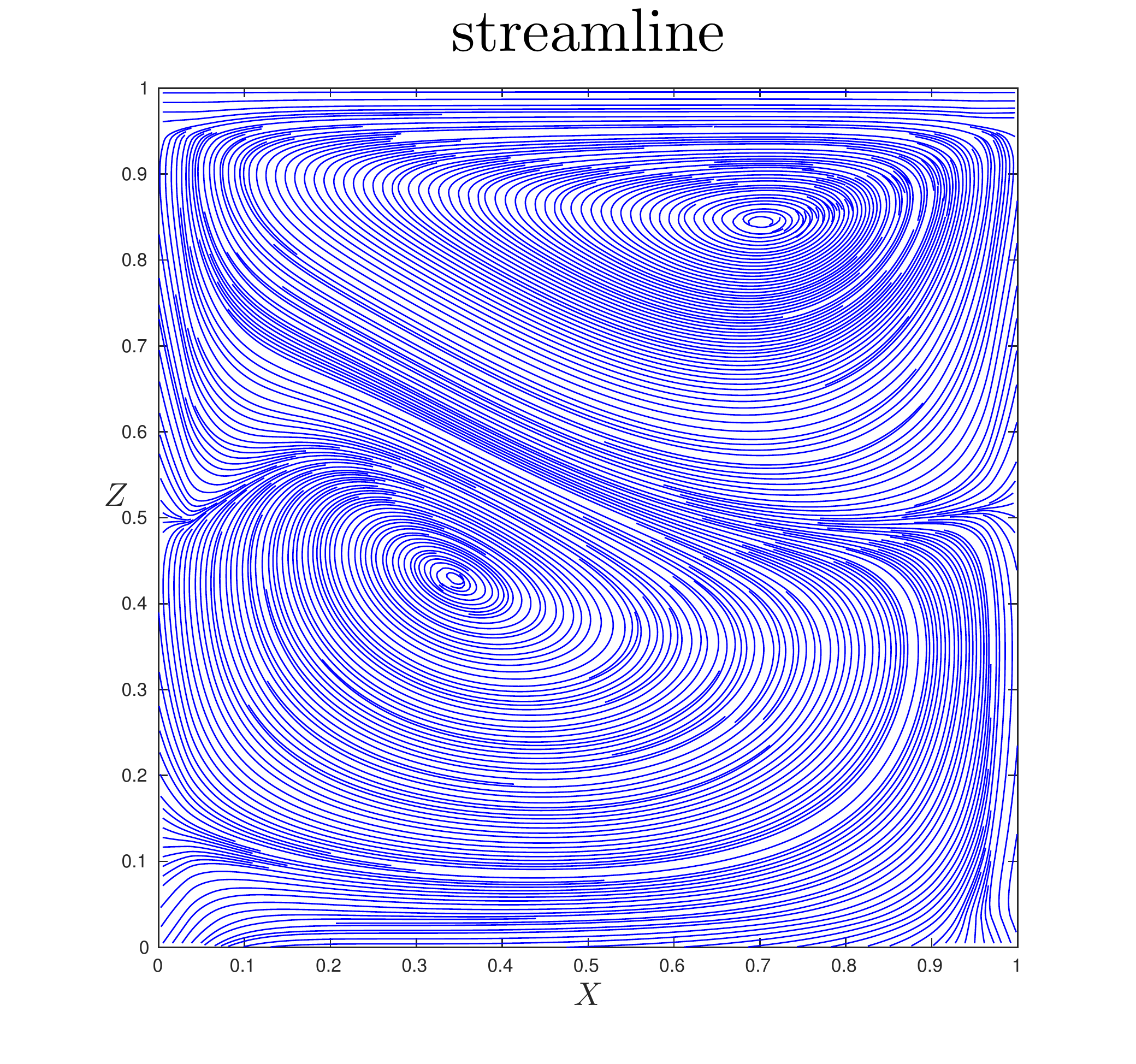,height=4in,width=4in}
  \end{tabular}
  \caption{Streamline of $\bu_h$ at $y=0.5$, $T=4.0$.
    Computed by $\bX_h^2 \times \bQ_h^2 \times M_h^1$
    on a uniform mesh with $M=\frac{1}{32}$, $\tau = 0.01$.
    (Example \ref{example-3d-cavity-flow})}
  \label{streamline}
\end{figure}
\begin{figure}[ht]
  \centering
  \begin{tabular}{c}
\epsfig{file=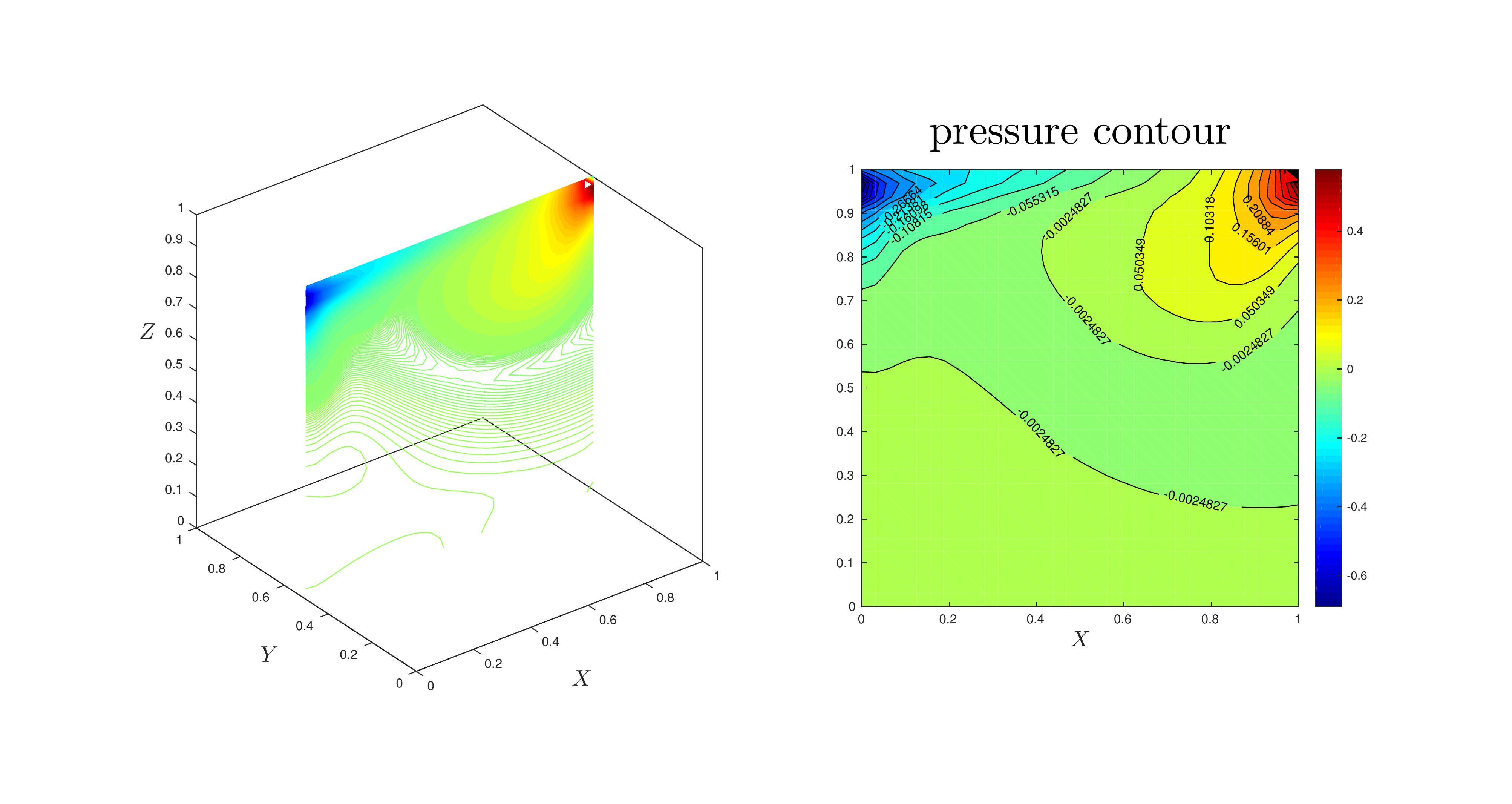,height=3.5in,width=6.0in}
  \end{tabular}
  \caption{Pressure contour of $p_h$ at $y=0.5$, $T=4.0$.
    Computed by $\bX_h^2 \times \bQ_h^2 \times M_h^1$
    on a uniform mesh with $M=\frac{1}{32}$, $\tau = 0.01$.
    (Example \ref{example-3d-cavity-flow})}
  \label{pressurecontour}
\end{figure}

As time evolves, we observe that the numerical solutions reach
the stationary state around $T=10.0$ with relative error 
\begin{align}
  \frac{\|\bu_h^n-\bu_h^{n-1}\|_{L^2}}{\|\bu_h^n\|_{L^2}}
  + \frac{\|p_h^n-p_h^{n-1}\|_{L^2}}{\|p_h^n\|_{L^2}}
  + \frac{\|\bB_h^n-\bB_h^{n-1}\|_{L^2}}{\|\bB_h^n\|_{L^2}}
  = 1.067 \times 10^{-4}.
  \label{error-indicator-lid}
\end{align}
We present the streamline and pressure plots at $T=10.0$ in Figures
\ref{streamline-steady-state} and \ref{pressurecontour-steady-state},
which are different with the results \cite[Example 5.4]{Hiptmair_LMZ}.
In our computation, the error \refe{error-indicator-lid} at $T=10.0$
still decreases as time evolves, however, with an extremely slow decay rate
(around $10^{-7}$ for each time step).
\begin{figure}[ht]
  \centering
  \begin{tabular}{c}
    \epsfig{file=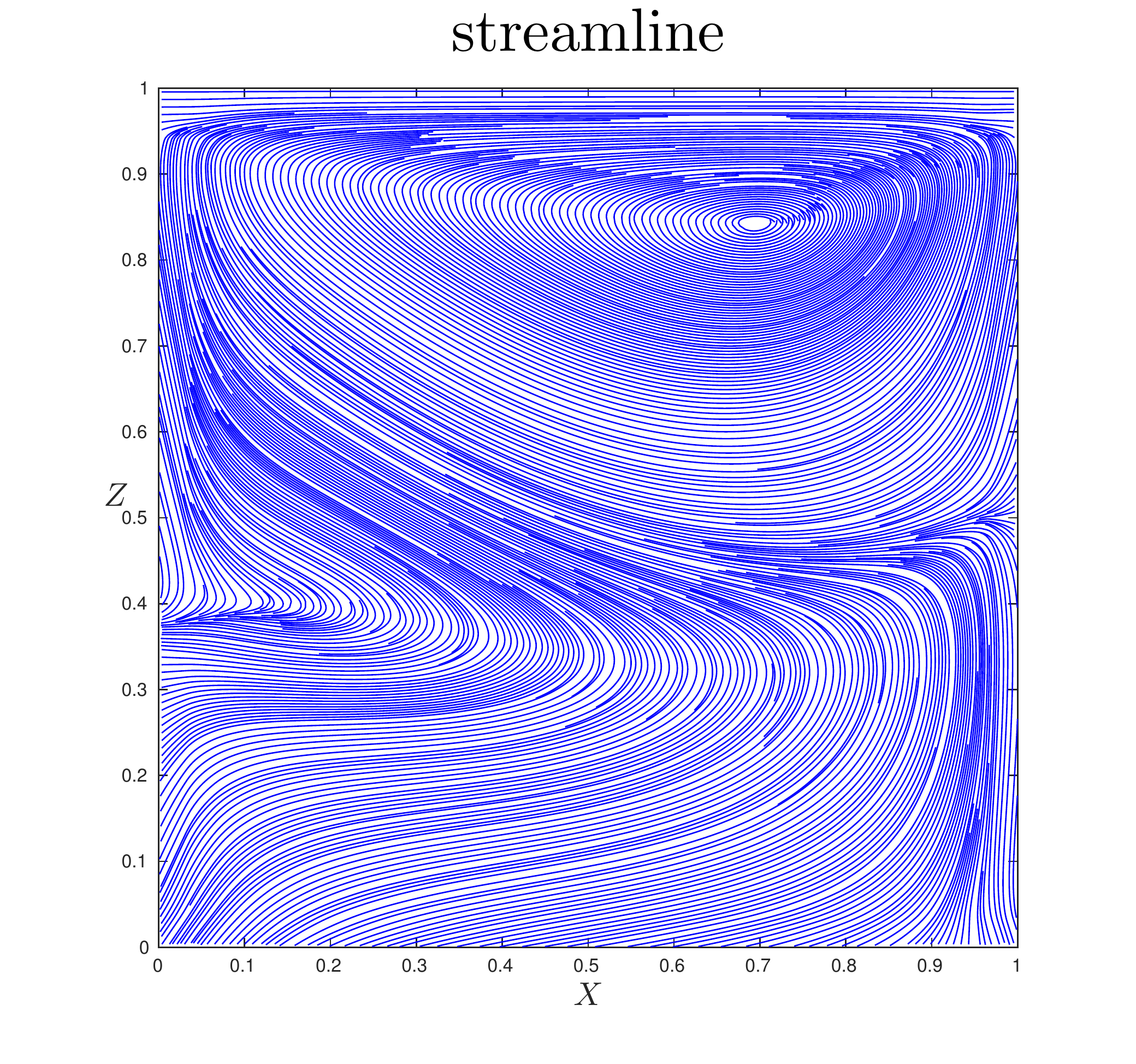,height=4in,width=4in}
  \end{tabular}
  \caption{Streamline of $\bu_h$ at $y=0.5$, $T=10.0$.
    Computed by $\bX_h^2 \times \bQ_h^2 \times M_h^1$
    on a uniform mesh with $M=\frac{1}{32}$, $\tau = 0.01$.
    (Example \ref{example-3d-cavity-flow})}
  \label{streamline-steady-state}
\end{figure}
\begin{figure}[ht]
  \centering
  \begin{tabular}{c}
\epsfig{file=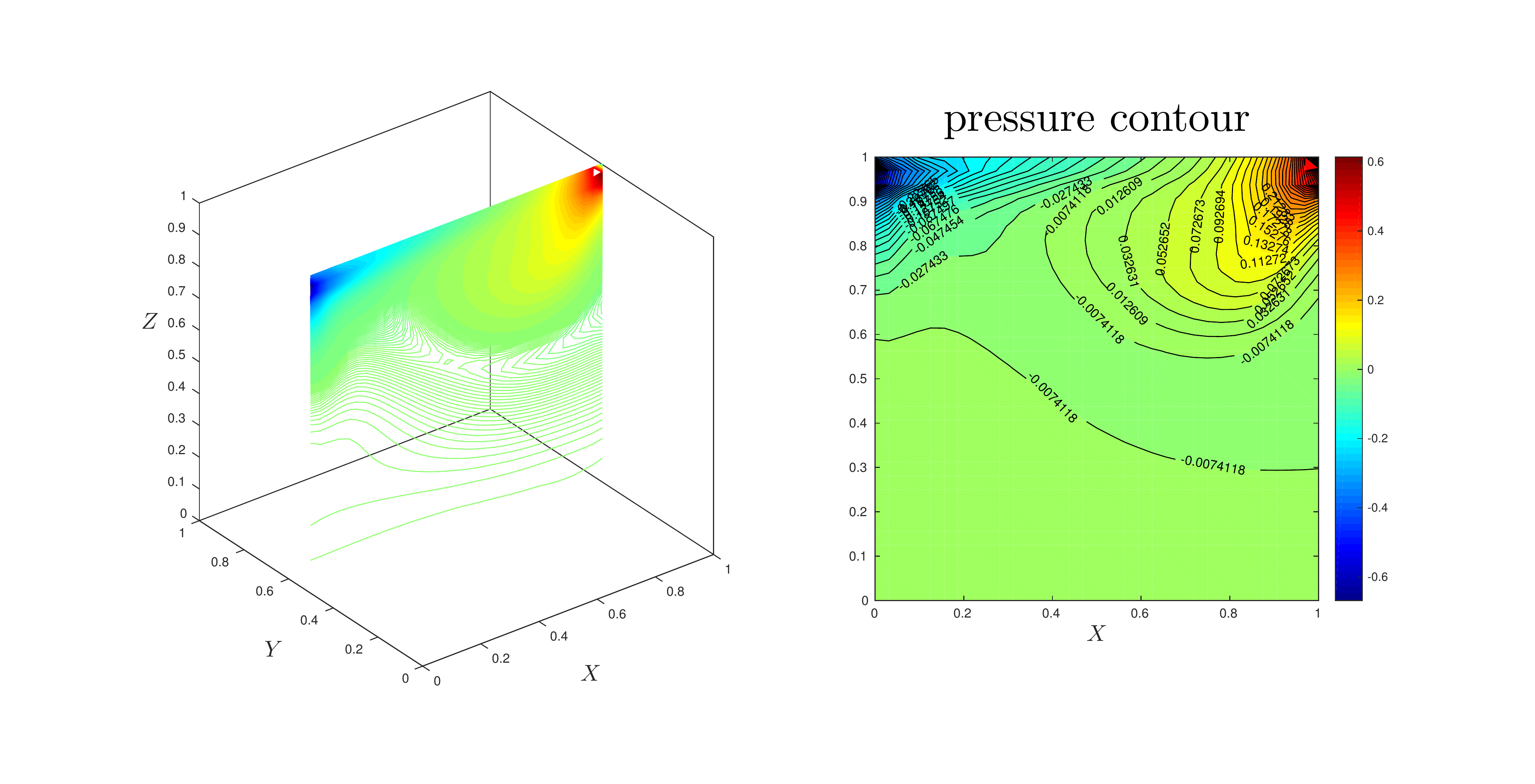,height=3.5in,width=6.0in}
  \end{tabular}
  \caption{Pressure contour of $p_h$ at $y=0.5$, $T=10.0$.
    Computed by $\bX_h^2 \times \bQ_h^2 \times M_h^1$
    on a uniform mesh with $M=\frac{1}{32}$, $\tau = 0.01$.
    (Example \ref{example-3d-cavity-flow})}
  \label{pressurecontour-steady-state}
\end{figure}

\end{example}

\section{Conclusions and remarks}
\label{numer-conclusion}
\setcounter{equation}{0}

The point of view that has been favored in some previous works is that
violating the divergence constraint at the discrete level
might lead to loss of energy preserving, which then
result in nonphysical numerical solutions.
However, we show in this paper that the energy preserving
still holds without satisfying the point-wise divergence free condition.
Furthermore, numerical experiments for the Hartmann flow and
the lid-driven cavity incompressible MHD flow 
demonstrate that the proposed scheme provide numerical results as good
as that obtained by schemes satisfying $\D \bB_h = 0$.

In this paper, we use a linearized backward Euler scheme,
which is first order accurate in the temporal direction.
The main reason is that, for the incompressible MHD problem,
the large storage is a crucial issue in the three-dimensional space.
There are many ways to design higher order integrator.
For instance, a second order linearized FEM is to look for
$(\bu_h^n, \bB_h^n, p_h^n)
\in \bX_h^k \times \bQ_h^{\widehat{k}} \times M_h^k$, such that
for any $(\bv_h, \bC_h, q_h) \in \bX_{h}^k \times {\bQ}_h^{\widehat{k}} \times {M}_{h}^k$
\begin{empheq}[left=\empheqlbrace]{align} 
  & \frac{1}{\tau}(\frac{3}{2}\bu_h^n - 2\bu_h^{n-1} - \frac{1}{2}\bu_h^{n-2}, \bv_h)
  + \frac{1}{R_e} (\,\nabla\, {\bu}_h^n \,,\, \nabla \bv_h)
  + \frac{1}{2} \big[ (\widehat{\bu}_h^{n} \cdot \nabla \, {\bu}_h^n, \bv_h)
    - (\widehat{\bu}_h^{n} \cdot \nabla \, \bv_h \,,\, {\bu}_h^n ) \big]
  \nn \\
  & \qquad \qquad \quad \!
  - ({p}_h^n, \nabla \cdot \bv_h)
  - S_c (\C \, {\bB}_h^{n} \, \times \widehat{\bB}_h^{n} \, , \, \bv_h)
  = (\mathbf{f}^{n}, \bv_h),
  \nn \\[3pt]
  & \frac{1}{\tau}(\frac{3}{2}\bB_h^n - 2\bB_h^{n-1}
  + \frac{1}{2}\bB_h^{n-2} , \bC_h )
  + \frac{S_c}{R_m} (\C \, {\bB}_h^n, \C \, \bC_h)
  - S_c ({\bu}_h^n \times \widehat{\bB}_h^{n}, \C \, \bC_h)
  = 0\,,
  \nn \\[3pt]
  & (\nabla \cdot {\bu}_h^n \, , \, q_h) = 0,
  \nn
\end{empheq}
where $\widehat{\bu}_h^{n} = 2{\bu}_h^{n-1} - {\bu}_h^{n-2}$
and  $\widehat{\bB}_h^{n} = 2{\bB}_h^{n-1} - {\bB}_h^{n-2}$.
For this three-level method, the linearized backward Euler FEM
\refe{fem-1}-\refe{fem-3} can be used to compute $(\bu_h^1, \bB_h^1, p_h^1)$.
We shall also remark that,
an $L^2$ error estimates of $O(\tau^2 + h^s)$ can be obtained by
a similar analysis, provided enough temporal regularity.
We focus on the homogeneous boundary condition \refe{bc-u}-\refe{bc-mag}
in this paper. 
It should e noted that the boundary condition \refe{bc-mag-alternate}
can be easily implemented with the {N\'ed\'elec} edge element.
All the theoretical results can be extended to models with
more complicated boundary conditions.

\section*{Acknowledgment}
The authors would like to thank Mr.~Ben Dai and Prof.~Junhui Wang
for help on conducting the numerical experiments.


\end{document}